\documentclass[leqno, 12pt]{article}
\usepackage{url}
\usepackage{pslatex}
\usepackage{amsmath,amsthm,amssymb,amsfonts, stmaryrd}
\usepackage{fancybox}
\usepackage{color, graphicx}
\usepackage[T1]{fontenc} 
\usepackage[usenames,dvipsnames,svgnames,table]{xcolor}
\newtheorem{theorem}{Theorem}[section]

\newtheorem{corollary}[theorem]{Corollary}

\newtheorem{definition}[theorem]{Definition}
\newtheorem{example}[theorem]{Example}

\newtheorem{lemma}[theorem]{Lemma}
\newtheorem{notation}[theorem]{Notation}

\newtheorem{proposition}[theorem]{Proposition}
\newtheorem{question}[theorem]{Question}
\newtheorem{remark}[theorem]{Remark}

\newcommand{\Z}{\mathbb Z}

\newcommand{\N}{\mathbb N}
\newcommand{\R}{\mathbb R}

\DeclareMathOperator{\vm}{vm}
\DeclareMathOperator{\ap}{Ap}

\DeclareMathOperator{\len}{len}
\DeclareMathOperator{\depth}{depth}
\DeclareMathOperator{\wt}{wt}

\newcommand{\vs}[1]{\langle #1 \rangle}

\usepackage{tikz}
\usetikzlibrary{calc}

\begin{document}

\title{A graph-theoretic approach to Wilf's conjecture}

\author{Shalom Eliahou}
\date{}

\maketitle

\begin{abstract}
Let $S \subseteq \N$ be a numerical semigroup with multiplicity $m = \min(S \setminus \{0\})$ and conductor $c=\max(\N \setminus S)+1$. Let $P$ be the set of primitive elements of $S$, and let $L$ be the set of elements of $S$ which are smaller than $c$. A longstanding open question by Wilf in 1978 asks whether the inequality $|P||L| \ge c$ always holds. Among many partial results, Wilf's conjecture has been shown to hold in case $|P| \ge m/2$ by Sammartano in 2012. Using graph theory in an essential way, we extend the verification of Wilf's conjecture to the case $|P| \ge m/3$. This case covers more than $99.999\%$ of numerical semigroups of genus $g \le 45$.
\end{abstract}

\begin{quote} \textit{Keywords and phrases.} Numerical semigroup; Apéry set; loopy graph; vertex-maximal matching; normality number; downset.
\end{quote}

\section{Introduction}\label{section intro}

Denote $\N=\{0,1,2,3,\dots\}$ and $\N_+=\N\setminus \{0\}=\{1,2,3,\dots\}$. For $a,b \in \Z$, let $[a,b[=\{z \in \Z \mid a \le z < b\}$ and $[a,\infty[=\{z \in \Z \mid a \le z\}$ denote the integer intervals they span. A \emph{numerical semigroup} is a subset $S \subseteq \N$ containing $0$, stable under addition and with finite complement in $\N$. Equivalently, it is a subset $S \subseteq \N$ of the form $S = \vs{a_1,\dots,a_n}=\N a_1 + \dots + \N a_n$ where $\gcd(a_1,\dots,a_n)=1$. The set $\{a_1,\dots,a_n\}$ is then called a \emph{system of generators} of $S$, and the smallest such $n$ is called the \emph{embedding dimension} of $S$.

For a numerical semigroup $S$, its \emph{gaps} are the elements of $\N \setminus S$, its \emph{genus} is $g=|\N \setminus S|$, its \emph{multiplicity} is $m = \min S^*$ where $S^*= S \setminus \{0\}$, its \emph{Frobenius number} is $f = \max \Z\setminus S$ and its \emph{conductor} is $c=f+1$. Thus $[c,\infty[ \, \subseteq S$ and $c$ is minimal for this property. As in~\cite{E}, we denote $L=S \cap [0,c[$. 

\smallskip
We partition $S^*$ as $S^*=P \sqcup D$, where $D=S^*+S^*=\{x+y \mid x,y \in S^*\}$ is the set of \emph{decomposable} elements of $S^*$, and $P= S^* \setminus D$ is the set of \emph{primitive elements} of $S^*$. As easily seen, $P$ is finite since $P \subseteq [m,c+m[$. Moreover $S = \vs{P}$ since every element of $S^*$ is a sum of primitive elements, and $P$ is the unique \emph{minimal system of generators} of $S$. Thus $|P|$ equals the embedding dimension of $S$. 

\smallskip
In 1978 Wilf asked, in equivalent terms, whether the inequality 
\begin{equation}\label{wilf conjecture}
|P||L| \ge c
\end{equation}
always holds~\cite{W}. Wilf's conjecture, as it is now known, has been verified in several cases, including when $|P| \le 3$, or $c \le 3m$, or $m \le 18$, or $|L| \le 12$, or $|P| \ge m/2$. See Delgado~\cite{D2} for an extensive recent survey of partial results on Wilf's conjecture, and \cite{Br08, BGOW, D, DM, E, EF, FGH, FH, K1, MS, Sa, Se, Sy, W} for some relevant papers. The verification in case $|P| \ge m/2$ is due to Sammartano~\cite{Sa} in 2012. Our purpose in this paper is to extend it to the case $|P| \ge m/3$.

\begin{theorem}\label{thm main} Let $S$ be a numerical semigroup with multiplicity $m$ and minimal generating set $P$. If $|P| \ge m/3$ then $S$ satisfies Wilf's conjecture.
\end{theorem}

This result was first presented in 2017 at a conference in Ume\r{a}~\cite{E2}. The present proof is a streamlined version of the original unpublished one. 

As later noted by Manuel Delgado, who attended the Ume\r{a} conference, an overwhelming majority of numerical semigroups satisfies the condition of Theorem~\ref{thm main}. Specifically, among all $23\,022\,228\,615$ numerical semigroups of genus $g \le 45$, the proportion of those satisfying $|P| \ge m/3$ exceeds $99.999\%$. In addition, Delgado discovered that the condition of Theorem~\ref{thm main} is well suited to efficiently trim the tree of numerical semigroups while probing certain open problems concerning them~\cite{D3}. In particular, this will lead to significant advances on the verification of Wilf's conjecture by computer. While the first such major effort reached genus $g = 50$~\cite{Br08}, and the current published verification record stands at genus $g = 60$~\cite{FH}, Delgado and Fromentin \emph{have now verified Wilf's conjecture up to genus $g = 80$}, and aim to reach genus $g = 100$ before publishing their result~\cite{DF}.

\subsection{Contents}

In Section~\ref{section depth}, we introduce the depth and total depth functions on a numerical semigroup. In Section~\ref{section graph}, we construct a map $S \mapsto G(S)$ associating to every numerical semigroup $S$ a finite graph $G(S)$ whose properties play a key role in this paper. Those properties, combining algebra and graph theory, are developed in Section~\ref{section properties}. Section~\ref{section main} is devoted to proving Theorem~\ref{thm main}. In the last Section~\ref{section equivalence}, we take a closer look at the map $S \mapsto G(S)$ by considering its range and fibers.

\section{The depth functions $\delta$ and $\tau$}\label{section depth}

Throughout this section, let $S \subseteq \N$ be a numerical semigroup with multiplicity $m$ and conductor $c$. 

\begin{definition}
The \emph{depth} of $S$ is the integer $q = \lceil c/m \rceil$. We denote it by $\depth(S)$. 
\end{definition}

See also~\cite{EF2}. More generally, we define the \emph{depth function} $\delta \colon S \to \Z$ on $S$ as follows.

\begin{definition} For all $x \in S$, let $\delta(x) \in \Z$ denote the unique integer such that 
$$
x + \delta(x)m \ \in \ [c,c+m[.
$$
We call $\delta(x)$ the \emph{depth} of $x$.
\end{definition}
For instance, assuming $S \not=\N$, the elements of $[c,c+m[$ have depth $0$, those in $[c+m, \infty[$ have negative depth while those in $S \cap [0,c[$ have positive depth. The largest depth in $S$ is attained by $0$, namely $\delta(0) = \depth(S) = \lceil c/m \rceil$. 
\begin{notation}\label{S_i} Let $q = \depth(S)=\lceil c/m \rceil$. We set $\rho=cm-q$. Thus $\rho \in [0,m[$ and $c=qm-\rho$.
\end{notation}

As in~\cite{E}, we denote
\begin{equation}\label{S_i with rho}
S_i = S \cap [im-\rho, im+m-\rho[
\end{equation}
for all $i \ge 0$. This yields the partition $S = \bigsqcup_{i \ge 0} S_i$. In particular, we have $S_0=\{0\}$, $m \in S_1$ and $c \in S_q$. More generally, we have
\begin{equation}\label{S_i with delta}
S_i = \{x \in S \mid \delta(x)=q-i\}
\end{equation}
as easily verified. Note also the equality
\begin{equation}\label{subset L}
L=S_0 \sqcup S_1 \sqcup \dots \sqcup S_{q-1}.
\end{equation}
The following was shown in~\cite{E}. Its verification is straightforward.

\begin{proposition}  Let $S$ be a numerical semigroup. For all $0 \le i \le j$ such that $j \ge 1$, we have
\begin{equation}\label{addition rule}
S_i+S_j \subset S_{i+j-1} \sqcup S_{i+j} \sqcup S_{i+j+1}.
\end{equation}
Moreover, if $\rho=0$ then
\begin{equation}\label{rho=0}
S_i+S_j \subseteq S_{i+j} \sqcup S_{i+j+1}. 
\end{equation}
\end{proposition} 

\smallskip
These set addition properties may be translated in terms of the depth function $\delta$ as follows. The rightmost inequality will be used throughout the paper.

\begin{proposition}\label{delta x y} Let $S$ be a numerical semigroup of depth $q \ge 1$. For all $x,y \in S$, we have
\begin{equation}\label{delta addition}
\delta(x+y)+q+1 \ \ge \ \delta(x)+\delta(y) \ \ge \ \delta(x+y)+q-\min(\rho,1).
\end{equation}
\end{proposition}
\begin{proof} As observed in \eqref{S_i with delta}, for all $x \in S$ we have 
$$x \in S_i \iff \delta(x)=q-i.$$ 
Let $x,y \in S$, and assume $x \in S_i$, $y \in S_j$. Then $\delta(x)=q-i$, $\delta(y)=q-j$, and so $\delta(x)+\delta(y)-q=q-i-j$. The addition properties~\eqref{addition rule} and \eqref{rho=0} now yield
$$
q-i-j-1 \le \delta(x+y) \le q-i-j+\min(\rho,1),
$$
whence
$$
\delta(x)+\delta(y)-q-1 \le \delta(x+y) \le \delta(x)+\delta(y)-q+\min(\rho,1).
$$
This is equivalent to~\eqref{delta addition}, as desired.
\end{proof}

\begin{definition} Let $A \subset S$ be a finite subset. We define the \emph{total depth} of $A$ as
$$
\tau(A) \ = \ \sum_{x \in A} \delta(x).
$$
\end{definition}
In the sequel, we use graph-theoretical tools to estimate the total depth $\delta(X)$ of $X$, the set of nonzero Apéry elements of $S$, as a step towards proving Theorem~\ref{thm main}. The key idea is to exploit \eqref{delta addition} by  forming suitable pairs $\{x,y\}$ of elements of $X$.

\subsection{The number $W(S)$ and Apéry elements}

Let $S \subseteq \N$ be a numerical semigroup of multiplicity $m$ and conductor $c$. As above, we partition $S^*=P \sqcup D$ into primitive and decomposable elements, and we set $L = S \cap [0,c[$. We shall use the following notation from~\cite{E}.

\begin{notation} $W(S)=|P||L|-c$.
\end{notation}
Thus, Wilf's conjecture amounts to state that $W(S) \ge 0$ holds for every numerical semigroup $S$. In this paper, as in~\cite{E}, we focus on estimating $W(S)$ from below. For this purpose, we need the nonzero Apéry elements of $S$. The set $$\ap(S)=\{s \in S \mid s-m \notin S\},$$ 
called the \emph{Apéry set} of $S$\footnote{Or more precisely, the \emph{Apéry set of $S$ with respect to $m$.}}, is central in the theory of numerical semigroups. It has $m$ elements, one in each class mod $m$, actually its least member belonging to $S$. As is well known and easy to see, the smallest and largest elements of $\ap(S)$ are $0$ and $c+m-1$, respectively. The additive properties of $\ap(S)\setminus \{0\}$ play a key role in this paper.

\begin{notation} We denote by $X = \ap(S) \setminus \{0\}$ the set of nonzero Apéry elements.
\end{notation}

\begin{proposition}\label{on X} The following hold.
\begin{itemize} \vspace{-0.15cm}
\item $\delta(x) \ge 0$ for all $x \in X$. \vspace{-0.15cm}
\item $m=|P|+|X \cap D|$. \vspace{-0.15cm} 
\item $|L| \ = \ q + \tau(X)$. 
\end{itemize}
\end{proposition}
\begin{proof} \hspace{0mm}
\begin{itemize} \vspace{-0.15cm}
\item As $\max X = c+m-1$, it follows that $X \subset [m,c+m[$. The conclusion follows from the definition of $\delta$. \vspace{-0.15cm}
\item We have $|X|=|X\cap P|+|X \cap D|$. The definitions imply that $|X|=m-1$ and $P \setminus X=\{m\}$, so $|X \cap P| = |P|-1$. The stated formula follows. \vspace{-0.15cm}
\item Let $a \in L$ be minimal in its class mod $m$. Then either $a=0$ or $a \in X$. Moreover $a+im \in L$ if and only if $i \in [0,\delta(a)[$. Hence
$$
|L \cap (a+m\N)| \ = \ \delta(a).
$$
Now $\delta(0)=q$, so that $\tau(L \cap m\N)=q$. Summing over all $x \in X$, i.e. over all nonzero classes mod $m$, we cover all of $L$ and the claimed formula follows. \vspace{-0.15cm}
\end{itemize}
\end{proof}

\begin{corollary}\label{prop w(s)} We have $W(S) \ = \ |P|\tau(X)-|X \cap D|q+\rho$.
\end{corollary}
\begin{proof} By definition, $W(S)=|P||L|-c=|P||L|-qm+\rho$. Since $|L|=q+\tau(X)$ and $m=|P|+|X \cap D|$ by Proposition~\ref{on X}, the stated formula follows.
\end{proof}
Our proof strategy for Theorem~\ref{thm main} will be to use graphs to estimate $\tau(X)$ from below using~\eqref{delta addition} and simultaneously estimate $|X \cap D|$ from above, thereby leading to show $W(S) \ge 0$ for the numerical semigroups under consideration. For this purpose, the following considerations will be useful. First, here is an analogue, in additive notation, of the notion of proper divisor.
\begin{definition} Let $b \in S^*$. A \emph{summand} of $b$ is any $a \in S^*$ such that $b \in a+S^*$, i.e. such that there exists $s \in S^*$ with $b=a+s$.
\end{definition}
As a matter of notation, given $a,b \in S$, it is customary to write $a \preceq b$ whenever $b-a \in S$. The following additive property is well known and crucial. 

\begin{lemma}\label{lem apery summand} Let $x \in X \cap D$. If $x=a+b$ with $a,b \in S^*$, then $a,b \in X$. That is, any summand of a nonzero Apéry element is a nonzero Apéry element.
\end{lemma}
\begin{proof} If $a \notin X$, then $a=a'+m$ for some $a' \in S^*$. Hence $x=a'+b+m$, whence $x \notin X$ since $a'+b \in S^*$.
\end{proof}

\section{The associated graph}\label{section graph}

In this section, we define a map $S \mapsto G(S)$ associating to every numerical semigroup $S$ a finite graph $G(S)$. Properties of $G(S)$ will then be shown to have a direct bearing on the parameters $\tau(X)$ and $|X \cap D|$ involved in Corollary~\ref{prop w(s)} and hence on Wilf's conjecture.  

\begin{definition} Let $S \subseteq \N$ be a numerical semigroup. The graph $G=G(S)$ associated to $S$ is defined as follows.
\begin{itemize} \vspace{-0.1cm}
\item The edge set $E(G)$ consists of all subsets $\{x,y\} \subseteq X$ such $x+y \in X$. The equality $x=y$ is allowed. \vspace{-0.1cm}
\item The vertex set $V(G)$ consists of all endvertices of the edges. Thus, an element $x \in X$ belongs to $V(G)$ if and only if there exists $y \in X$ such that $x+y \in X$. 
\end{itemize}
\end{definition} \vspace{-0.1cm}

\begin{remark} More generally, one may associate a graph $G(A)$ to any finite (or not) subset $A$ of a monoid $(M,+)$. The edges of $G(A)$ are all subsets $\{x,y\} \subseteq A$ such that $x+y \in A$, and its vertices are all endvertices of the edges. This graph carries much information on the additive properties of $A$. For a numerical semigroup $S$, the graph $G(S)$ is obtained in this general form by taking $G(S) = G(A)$, where $A=X$ is the set of nonzero Apéry elements of $S$.
\end{remark}

By construction, the graph $G(S)$ has no isolated vertices. More generally, it follows from the definition that $G(S)$ is a \emph{loopy graph} as defined below.

\begin{definition}
A \emph{loopy graph} is a finite graph with no isolated vertices, no multiple edges but possibly with loops. 
\end{definition}

We shall further need the following definitions/notation.

\begin{definition}
In a loopy graph, an edge with equal endvertices is a \emph{loop}, otherwise it is a \emph{true edge}. A vertex is \emph{loopy} if it supports a loop, or \emph{nonloopy} otherwise. The \emph{loopy-complete graph} on $n$ vertices, denoted $LK_n$, is the graph obtained from the complete graph $K_n$
by attaching a loop to every vertex.
\end{definition}

\begin{notation}
For a loopy graph $G$, we denote by $\lambda(G)$ its number of loops. It coincides with its number of loopy vertices since $G$ has no multiple edges.
\end{notation}

For example, Figure~\ref{pics} displays $G(S)$ for $S = \vs{12, 13, 14, 15, 17, 19, 20, 21}$. Here $|P|=8$, $m=12$ and $X=\{13, 14, 15, 17, 19, 20, 21, 28, 30, 34, 35\}$. In particular, the three loopy vertices are $14,15,17$, exactly those $x \in X$ such that $2x \in X$.

\medskip
\begin{figure}[h]
\begin{center}
\begin{tikzpicture}[scale=0.9, rotate=180]
   \path (0:1.7cm) coordinate (P0);
   \path (1*72:1.7cm) coordinate (P1);
   \path (2*72:1.7cm) coordinate (P2);
   \path (3*72:1.7cm) coordinate (P3);
   \path (4*72:1.7cm) coordinate (P4);

   \path (P2) +(168:1.9cm) coordinate (B);
   \path (P3) +(198:1.95cm) coordinate (A);
   
\draw [fill] (P0) circle (0.08);
\draw [fill] (P1) circle (0.08);
\draw [fill] (P2) circle (0.08);
\draw [fill] (P3) circle (0.08);
\draw [fill] (P4) circle (0.08);

\draw [fill] (A) circle (0.08);
\draw [fill] (B) circle (0.08);

   \draw[thick] (P0) -- (P1) -- (P2) -- (P3) -- (P4) -- cycle;

\draw[thick] (P3)--(A);
\draw[thick] (P2)--(B);

\draw[thick, rotate=-90] (P0) to [out=180,in=180] ++(0,1.25) to [out=0,in=0] ++(0,-1.25);

\draw[thick, rotate=180] (P3) to [out=160,in=180] ++(0,1.25) to [out=0,in=20] ++(0,-1.25);

\draw[thick, rotate=78] (B) to [out=180,in=180] ++(0,1.25) to [out=0,in=0] ++(0,-1.25);

\end{tikzpicture}
\end{center}
\caption{The graph $G(S)$ associated to $S = \vs{12, 13, 14, 15, 17, 19, 20, 21}$.}
\label{pics}
\end{figure}
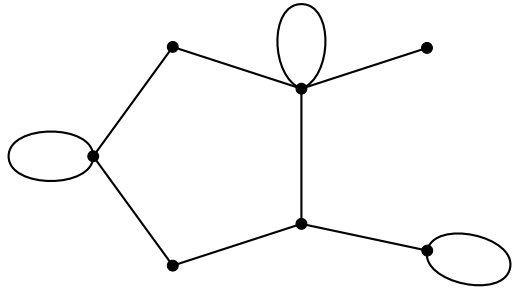

\subsection{Vertex-maximal matchings}\label{vertex-maximal}

Let $G=(V,E)$ be a loopy graph. A \emph{matching} $M$ in $G$ is a subgraph consisting of mutually nonadjacent edges. Loops are allowed in $M$.

\begin{definition} The \emph{vertex-maximal matching number of $G$} is the maximum number of vertices touched by a matching $M$ in $G$. We denote this number by $\vm(G)$. In formula:
$$
\vm(G) = \max_{M \subseteq G} |V(M)|
$$
where $M$ runs over all matchings of $G$. 
\end{definition}

\begin{definition}\label{active} A \emph{vertex-maximal matching of $G$} is a matching touching $\vm(G)$ vertices. An edge in $G$ is \emph{active} if it is contained in a vertex-maximal matching of $G$, and \emph{passive} otherwise. We denote by $E^+ \subseteq E$ the set of active edges.
\end{definition}

A loop needs not be active in general. However, a vertex-maximal matching contains all the loopy vertices, as easily seen. Moreover, we have $\vm(G) \ge \lambda(G)$, since any set of $\ell$ loops in $G$ is a matching with $\ell$ vertices.

\begin{proposition}\label{edge max} Let $G$ be a loopy graph with $\vm(G)=k$ and such that $G$ is edge-maximal for this property. Let $\ell=\lambda(G)$. Then $G$ contains $LK_\ell$.
\end{proposition}

\begin{proof} As mentioned above, every vertex-maximal matching in $G$ contains all of its $\ell$ loopy vertices\footnote{But again, not necessarily all of its \emph{loops}.}. Assume  that $x,y$ are nonadjacent loopy vertices. Then, as easily seen, adding the edge $\{x,y\}$ to $G$ does not increase $\vm(G)$. This contradicts the edge-maximality of $G$ with respect to $\vm(G)$. Hence $G \supseteq LK_\ell$.
\end{proof}

An interesting general question, with direct implications for the present approach to Wilf's conjecture, is the following.

\begin{question}\label{question} Given integers $n \ge k \ge 1$, let $G$ be a loopy graph on $n$ vertices and such that $\vm(G)=k$. What is the maximum number of edges allowed in $G$?
\end{question}
For instance, consider a loopy graph $G$ with $(n,k)=(5,4)$. While the non-complying graph $LK_5$ has $15$ edges, we show in Proposition~\ref{bound 10} that $G$ has at most $10$ edges, and this is optimal as witnessed by the complying graph $K_5$.

\medskip
For $n \ge k+2$ with $k \ge 2$ even, say $k=2r$, it might be that the optimal upper bound on $|E(G)|$ seeked in Question~\ref{question} is given by
$$
\binom{r+1}{2}+r(n-r).
$$
This number of edges is achieved by the complying graph $G = LK_r \vee \overline{K_{n-r}}$, the \emph{join}~\cite{CLZ} of $LK_r$ and the empty graph $\overline{K_{n-r}}$ on $n-r$ vertices. Recall that $G_1 \vee G_2$ is obtained by adding to $G_1 \sqcup G_2$ all possible edges between $V(G_1)$ and $V(G_2)$. 

A similar construction can be made for $k$ odd.

\subsection{The weight of edges}

Let $G=G(S)$ be the graph associated to a numerical semigroup $S \subseteq \N$. As usual, we denote by $D,X \subset S^*$ the sets of decomposable and nonzero Apéry elements, respectively.

\begin{definition} Let $e=\{x,y\} \in E(G)$. The \emph{weight} of $e$ is defined as $\wt(e)=x+y$. 
\end{definition}
By construction, this yields a map $\wt \colon E(G) \to X \cap D$.

\begin{proposition} The map $\wt \colon E(G) \to X \cap D$ is onto.
\end{proposition}
\begin{proof} For every $z \in X \cap D$, there exist $x,y \in X$ such that $z=x+y$. Thus $\{x,y\}$ is an edge of $G$ and has weight $z$.
\end{proof}
It follows that 
\begin{equation}\label{X cap D le E}
|X \cap D| \le |E(G)|.
\end{equation}
Here is a useful formula for the difference $|E(G)|-|X \cap D|$.

\begin{proposition}\label{prop fibers} We have
$$
|X \cap D| = |E(G)| - \sum_{z \in X \cap D} (|\wt^{-1}(z)|-1).
$$
\end{proposition}
\begin{proof} The fibers of $\wt$ constitute a partition of $E(G)$. Thus
$$
|E(G)| = \sum_{z \in X \cap D} |\wt^{-1}(z)|.
$$
Note that $|\wt^{-1}(z)| \ge 1$ for all $z \in X \cap D$ since $w$ is onto. Subtracting $1$ to each such summand yields
$$
|E(G)| = |X \cap D|+\sum_{z \in X \cap D} (|\wt^{-1}(z)|-1). \qedhere
$$
\end{proof}

In particular, the larger $|V \cap D|$ is, the farther away $|X \cap D|$ will be from $|E(G)|$. For instance, if there is at least one fiber of cardinality more than $1$, then $|X \cap D| < |E(G)|$.

\begin{remark}
If all edge weights are distinct, then $\wt$ is a bijection and hence $|X \cap D|=|E(G)|$. 
\end{remark}

\begin{lemma} Distinct adjacent edges have distinct weights. Similarly, distinct loops have distinct weights. 
\end{lemma}
\begin{proof} Distinct adjacent edges are of the form $\{x,y\}, \{x,z\}$ with $y \not= z$, whence $x+y \not= x+z$. Distinct loops are of the form $\{x,x\}, \{y,y\}$ with $x \not= y$, implying $2x \not= 2y$.
\end{proof}

\subsection{Normal and weak edges}

We use the same notation as above.

\begin{lemma} Let $\{x,y\}$ be an edge in $G$. Then $\delta(x)+\delta(y) \ge q-\min(\rho,1)$.
\end{lemma}
\begin{proof} We have $x+y \in X$ by hypothesis. The inequality now directly follows from~\eqref{delta addition} and Proposition~\ref{on X}.
\end{proof}

\begin{definition} An edge $\{x,y\}$ in $G$ is \emph{weak} if $\delta(x)+\delta(y) = q-1$, and \emph{normal} otherwise, i.e. if $\delta(x)+\delta(y) \ge q$.
\end{definition}

\begin{remark} If $\rho=0$ then all edges of $G$ are normal. This follows from the above lemma.
\end{remark}

\begin{notation} We denote by $E_0(G)$ and $E_1(G)$ the set of \emph{weak} and \emph{normal} edges of $G$, respectively. Thus
$$
E(G) = E_0(G) \sqcup E_1(G).
$$
\end{notation}

\begin{lemma}\label{weak edge} If $\{x,y\} \in E_0(G)$, then $\delta(x+y)=0$.
\end{lemma}
\begin{proof} Indeed, by hypothesis we have $x+y \in X$ and $\delta(x)+\delta(y)=q-1$. The former implies $\delta(x+y) \ge 0$ by Proposition~\ref{on X}, and the latter implies $\rho \ge 1$ and $\delta(x+y) = 0$ by \eqref{delta addition}.
\end{proof}

\begin{proposition}\label{prop X0} Let $S \subseteq \N$ be a numerical semigroup. Let 
\begin{equation}\label{X0}
X_0=\{z \in X \cap D \mid \exists x,y \in X, \ z = x+y, \ \delta(x)+\delta(y) \ = \ \delta(z) + q -1\}.
\end{equation}
Then $|X_0| \le \rho$.
\end{proposition}
\begin{proof} Let $z = x+y \in X_0$, and assume $x \in S_i, y \in S_j$. Then $\delta(x)+\delta(y) \ = \ \delta(z) + q -1$ if and only if $z \in S_{i+j-1}$. Now, by the definition of the $S_i$, we have
$$(S_i+S_j) \cap S_{i+j-1} \subseteq [(i+j)m-2\rho,(i+j)m-\rho[.
$$
Thus, the only classes mod $m$ for which such a deficit may occur are those in $[-2\rho,-\rho[$. And since there is only one element of $X$ per class mod $m$, the statement follows.
\end{proof}

\begin{corollary}\label{rho} We have $\rho \ge |\wt(E_0(G))|$.
\end{corollary}
\begin{proof} Let $X_0 \subseteq X \cap D$ be as defined in \eqref{X0}. It suffices to show
\begin{equation}\label{wt E0 subset X0}
\wt(E_0(G)) \subseteq X_0,
\end{equation}
and the conclusion will follow from Proposition~\ref{prop X0}. Let $e=\{x,y\} \in E_0(G)$. Then $\delta(x)+\delta(y)=q-1$ by hypothesis. Let $z=\wt(e)=x+y$. Then $z \in X \cap D$ by definition, and $\delta(z) = 0$ by Lemma~\ref{weak edge}. Therefore $z \in X_0$ and we are done.
\end{proof}

\subsection{The normality number}

We keep using the same notation as above.
\begin{definition} The \emph{normality number} of the graph $G=G(S)$ is defined as
$$
\nu=\nu(G) = \max_{M \subseteq G}\ \#\{\textrm{endvertices of all normal edges in } M\},
$$
where $M$ runs over all vertex-maximal matchings in $G$. Thus $0 \le \nu \le \vm(G)$.
\end{definition}

Recall from Section~\ref{vertex-maximal} that an edge is \emph{active} if it belongs to a vertex-maximal matching, and that we denote by $E^+ \subseteq E$ the subset of active edges. The partition $E=E_0 \sqcup E_1$ into weak and normal edges induces a corresponding partition on $E^+$. 
\begin{notation} We denote by $E_0^+ \subseteq E_0$ the subset of active weak edges, and by $E_1^+ \subseteq E_1$ the subset of active normal edges. Thus $E^+=E_0^+ \sqcup E_1^+$.
\end{notation}

The interest of this partition is that \emph{only active edges} are actually involved in the definition of the normality number $\nu(G)$. That is, we have
\begin{equation}\label{on nu}
\nu(G) = \max_{M \subseteq G}\ \#\{\textrm{endvertices of } E(M) \cap E_1^+\},
\end{equation}
where $M$ runs over all vertex-maximal matchings in $G$.

\subsection{A lower bound on $\tau(X)$}

We now have all the ingredients at hand to formulate our key lower bound on $\tau(X)$ and hence on $W(S)$. We keep using the same notation as above.

\begin{theorem}\label{thm formula} Let $G=G(S)$, $n=|V(G)|$ and $k=\vm(G)$. Then
$$
\tau(X) \ge \big(k(q-1)+\nu\big)/2 +(n-k).
$$
\end{theorem}
\begin{proof} Let $M \subseteq G$ be a vertex-maximal matching, and set $V_M = V(M)$. Thus $|V_M|=k$. Moreover, by \eqref{on nu}, we may assume that the number of vertices touched by the normal edges of $M$ is maximal, i.e. is equal to $\nu=\nu(G)$. 

We have $\tau(X) \ge \tau(V)$ since $V \subset X$. We now evaluate $\tau(V)$ from below. Let $\overline{V}_M=V \setminus V_M$. Then $|\overline{V}_M|=n-k$. We have $\tau(V)=\tau(V_M)+\tau(\overline{V}_M)$. Since $V \subset L$ and since $\delta(a) \ge 1$ for all $a \in L$, we have 
$$\tau(\overline{V}_M) \ge |\overline{V}_M|=n-k.$$
We now estimate $\tau(V_M)$. For that, we need to count the edges of $M$ by distinguishing the nonloops and the loops, and the weak and the normal ones. Let $r_0,t_0$ denote the number of weak nonloops and loops in $M$, respectively. Similarly, let $r_1,t_1$ denote the number of normal nonloops and loops in $M$, respectively. Thus
$$
k= 2(r_0+r_1)+t_0+t_1, \,\, \nu=2r_1+t_1.
$$

For every edge $\{x,y\}$ in $M$, we have $\delta(x)+\delta(y)=q-1$ if it is weak, while $\delta(x)+\delta(y) \ge q$ if it is normal. It follows that
\begin{eqnarray*}
\tau(V_M) & \ge & r_0(q-1)+r_1q + t_0(q-1)/2+t_1q/2 \\
& = & \big((2r_0+t_0+2r_1+t_1)(q-1)+2r_1+t_1\big)/2 \\
& = & \big(k(q-1)+\nu \big)/2.
\end{eqnarray*}
Summarizing, we have 
$$
\tau(X) \ge \tau(V) = \tau(V_M)+\tau(\overline{V}_M) \ge \big(k(q-1)+\nu \big)/2+(n-k). \qedhere
$$
 \end{proof}

\section{Properties of $G(S)$}\label{section properties}

Let $G(S)=G=(V,E)$ be the graph associated to the numerical semigroup $S$. Most results in this section, combining algebraic and graph-theoretic properties, will be used in Section~\ref{section main} to prove Theorem~\ref{thm main}.

\smallskip
Among the vertices in $V$, distinguishing between the primitive and the decomposable ones is crucial. Thus, we shall systematically consider the partition
$$
V = (V \cap P) \sqcup (V \cap D).
$$

In this context, we prefer using the more intuitive multiplicative notation, as the elements of $V \cap D$ are best viewed as \emph{monomials} in $V \cap P$. 

For instance, if $V \cap P=\{x_1,x_2\}$ and $V \cap D=\{2x_1, x_1+x_2, 2x_2, 3x_1\}$ in standard additive notation, we prefer to write $V \cap D = \{x_1^2, x_1x_2, x_2^2, x_1^3\}$. In this way, we can speak of divisors, multiples, antichains under divisibility, and so on. For instance, we find it more convenient to say ``$x_1$ divides $x_1x_2$'' rather than ``$x_1$ is a summand of $x_1+x_2$'' or write $x_1 \preceq (x_1+x_2)$ in standard additive notation.

\smallskip
More formally, let us rename our given additive numerical semigroup $S$ as $S_0$. We then embed $S_0$ in the one-variable polynomial ring $\R[Z]$, and more precisely in the semigroup ring $\R[S_0] \subseteq \R[Z]$, where
$$
\R[S_0]=\{\sum_{a \in S_0} \lambda_aZ^a \mid \lambda_a \in \R \textrm{ for all } a \in S_0 \textrm{ and } \lambda_a = 0 \textrm{ for almost all }a \}.
$$
We then set 
$S=\{Z^a \mid a \in S_0\}.$
It is a multiplicative submonoid of $\{Z^n \mid n \in \N\}$ with finite complement and neutral element $Z^0=1$. We have a monoid isomorphism
\begin{equation}\label{iso}
\varphi \colon S_0 \to S
\end{equation}
defined by $\varphi(a)=Z^a$ and satisfying $\varphi(a+b)=\varphi(a)\varphi(b)$ for all $a,b \in S_0$.  We will refer to $S$ as a \emph{numerical semigroup in multiplicative notation}. 

\subsection{Switching to multiplicative notation}

Thus, from now on in this section, $S$ is a numerical semigroup in multiplicative notation, arising from its additive counterpart $S_0 \subseteq \N$ via the isomorphism $\varphi$ in \eqref{iso}. We denote $S^*=S \setminus \{1\}$. All other usual notions related to $S_0$, such as the multiplicity, the conductor, the subsets $L,P,D,X,V$ and so on, are transported via $\varphi$ to $S$ without changing notation.

For clarity, let us rewrite the weight of edges of $G=G(S)$ in multiplicative notation. The weight map $\wt \colon E(G) \to X \cap D$ is then defined as follows: for any edge $\{x,y\} \in E(G)$, we set
 $$
 \wt(\{x,y\}) = xy.
 $$
 Note that $xy \in X \cap D$ by construction.
 
A word of caution is needed here. The decomposition of an element $z \in X \cap D$ as a product of primitive elements is not unique in general. That is, $z$ may be represented by several formally distinct monomials in $P$. On the other hand, we do have simplification properties such as $$x^2=y^2 \Rightarrow x=y \,\,\textrm{ and }\,\, xz=yz\Rightarrow x=y$$
for all $x,y,z \in S$, as follows from the analogous additive properties in $S_0 \subseteq \N$.
 
\subsection{Downsets}

As above, let $S$ denote a numerical semigroup in multiplicative notation.

\begin{definition} Let $u \in S^*$. A \emph{proper factor} of $u$ is an element $v \in S^*$ such that $v \not=u$ and $v$ divides $u$, i.e. such that there exists $v' \in S^*$ satisfying $u=vv'$. 
\end{definition}

\begin{definition} A \emph{downset} in $S^*$ is a subset $I \subseteq S^*$ which is stable under taking proper factors. That is, if $u \in I$ and if $v \in S^*$ is a proper factor of $u$, then $v \in I$.
\end{definition}

The following lemma is a restatement of Lemma~\ref{lem apery summand} in the present context.

\begin{lemma}\label{lem downset} The subset $X \subset S^*$ is a downset. \hfill $\Box$
\end{lemma}

\begin{lemma}\label{V and X cap D} The set $V$ of vertices of $G$ is a downset. It coincides with the set of proper factors of all elements of $X \cap D$.
\end{lemma}
\begin{proof} Let $x \in V$. Then there exists $y \in X$ such that $xy \in X$ and so $\{x,y\} \in E$. Actually $xy \in X \cap D$ and $x$ is a proper factor of $xy$. If $x'$ is a proper factor of $x$, then $x'y$ is a proper factor of $xy$, hence it belongs to $X$ since $X$ is a downset, hence $\{x',y\} \in E$. This implies $x' \in V$. Therefore $V$ is a downset, as claimed. Let now $z \in X \cap D$, and let $x \in S^*$ be a proper factor of $z$. Let $y=z/x$. Then $x,y \in X$ by Lemma~\ref{lem downset} and $\{x,y\} \in E$. Hence $x \in V$, as desired.
\end{proof}

Given a vertex $x \in V$, we denote as usual by $N_G(x) \subseteq V$ its set of neighbors, i.e.
$$
N_G(x) = \{y \in X \mid xy \in X\} = \{y \in V \mid xy \in X\}.
$$
As usual, the \emph{degree} of vertex $x$ is defined as $\deg(x) = |N_G(x)|$.

\begin{lemma}\label{neighbor} Let $u \in V$. Then $N_G(u)$ is a downset.
\end{lemma}
\begin{proof} We have $uv \in X$ since $v \in N_G(u)$. Let $w$ be a proper factor of $v$. Then $w \in V$ and $v=v'w$ for some $v' \in V$. Hence $uv'w \in X$, implying $uw \in X$, implying in turn $w \in N_G(u)$.
\end{proof}

\subsection{More vertex properties}

\begin{lemma}\label{P and V} We have $|P| \ge |V \cap P|+1$.
\end{lemma}
\begin{proof} Indeed, with $m$ denoting as usual the multiplicity of $S$, we have $m \in P \setminus V$ since $m \notin X$.
\end{proof}

The next result helps locate in $V$ the proper factors of the vertices in $V \cap D$, if any. 

\begin{proposition}\label{degrees} Let $v_1 \not= v_2 \in V$. If $v_1$ divides $v_2$, then $\deg(v_1) > \deg(v_2)$.
\end{proposition}

\begin{proof} Let $w \in V$ be such that $v_2=v_1w$. Let $t = \deg(v_2)$ and denote $N_G(v_2) = \{z_1,\dots,z_t\}$.
Since $z_iv_2=z_iwv_1 \in X$ for all $i$ by hypothesis, and since $X$ is a downset, it follows that
$$
\{w,z_1,\dots,z_t, z_1w, \dots, z_tw\} \subseteq N_G(v_1).
$$
That set is of cardinality at least $t+1$ since $w,z_1w,\dots,z_tw$ are pairwise distinct. Whence $\deg(v_1) \ge t+1$, as desired.
\end{proof}

\begin{corollary}\label{V subset P} All vertices in $G$ of maximal degree belong to $V \cap P$. Moreover, for any $r \ge 1$, the subset of vertices of $G$ of degree $r$ forms an antichain under divisibility. \hfill $\Box$
\end{corollary}

\subsection{On loopy and nonloopy vertices}
\begin{definition} Let $z \in S^*$. We define the \emph{length} of $z$ to be the largest integer $t \ge 1$ such that $z=x_1\dots x_t$ with $x_1,\dots,x_t \in S^*$. We then write $t = \len(z)$. 
\end{definition}
In particular, $\len(z)=1$ if and only if $z \in P$. Since $X$ is a downset, it follows that if $z \in X$, then $\len(z)$ coincides with the largest integer $t \ge 1$ such that $z=x_1\dots x_t$ with $x_1,\dots,x_t \in X$. 

\begin{proposition}
All vertices in $V \cap D$ of maximal length are nonloopy.
\end{proposition}
\begin{proof}
Let $u \in V \cap D$ be of maximal length, say $t \ge 2$. Let $x \in V \cap P$ be a proper factor of $u$, say $u=xv$ with $v \in X$. Assume for a contradiction that $u$ is loopy. Then $u^2 \in X$. Since $u^2=xvu$ and $v \in X$, it follows that $xu \in V \cap D$ and $\len(xu) \ge t+1$. This contradicts the maximality of $t$. Therefore $u$ is a nonloopy vertex of $G$, as claimed.
\end{proof}

\begin{corollary} If all vertices in $G$ are loopy, then $V \cap D = \emptyset$, i.e. $V \subset P$. \hfill $\Box$
\end{corollary}

\begin{lemma}\label{nonloopy} Let $y \in V$ be a nonloopy vertex. Then $y$ divides none of its neighbors in $G$.
\end{lemma}
\begin{proof} Let $z \in N_G(y)$ such that $z=yz'$ with $z' \not=1$. We have $yz \in X$ since $y,z$ are neighbors. Hence $y^2z' \in X$, implying $y^2 \in X$ and thus contradicting that $y$ is a nonloopy vertex.
\end{proof}

\begin{lemma}\label{factor of loopy} Every proper factor of a loopy vertex is loopy.
\end{lemma}
\begin{proof} Let $u \in V$ and assume that $u$ is loopy. Hence $u^2 \in X$. Let $v \in V$ be a proper factor of $u$. Since $X$ is stable under taking proper factors, it follows that $v^2 \in X$. Whence $v$ is loopy.
\end{proof}

\begin{lemma}\label{unique loopy} If $\lambda(G)=1$, then the unique loopy vertex $u \in V$ is primitive.
\end{lemma}
\begin{proof} We have $u^2 \in X$ since $u$ is loopy. If $u \in D$, then $u=ab$ with $a,b \in X$. Therefore $a^2 \in X$, so that $a$ is also a loopy vertex, and we are done since $a \not= u$.
\end{proof}

\subsection{More on $V \cap P$ and $V \cap D$}

\begin{proposition}\label{card V cap D = 1} We have $|V \cap D| \ge \deg(u)$ for all $u \in V \cap D$. If $V \cap D = \{u\}$, then $N_G(u)=\{x\}$ for some $x \in V \cap P$, and $u=x^2$. 
\end{proposition}
\begin{proof} Let $u \in V \cap D$. We have $u=wv$ for some $w \in V$. Let $t = \deg(u)$ and denote
$$
N_G(u) = \{z_1,\dots,z_t\}.
$$
Since $z_iu=z_iwv \in X \cap D$ for all $i$ by hypothesis, it follows that 
$$
\{z_1w, \dots, z_tw\} \subseteq V \cap D,
$$
whence $|V \cap D| \ge t$. Assume now $V \cap D = \{u\}$ with $u=wv$ as above. Since $|V \cap D|=1$, it follows from the above that $t=1$, whence $N_G(u)=\{z_1\}$. Thus $z_1wv \in X \cap D$, implying $\{z_1w, z_1v, wv\} \subseteq V \cap D$. Therefore $z_1w = z_1v = wv$, whence $z_1=w=v$ and $u=z_1^2$. Moreover $z_1 \in P$, for if $z_1$ had proper factors in $V$, this would imply $z_1 \in V \cap D$, contradicting the equality $V \cap D= \{z_1^2\}$.
\end{proof}

\begin{proposition}\label{large diff} We have $|X \cap D| \le |E(G)|-\deg(u)$ for all $u \in V \cap D$ such that $u \not= x^2$ with $x \in P$. 
\end{proposition}

\begin{proof} Let $u \in V \cap D$ be such that $u \not= x^2$ with $x \in P$. 
Let $x \in V \cap P$ be a primitive factor of $u$, so that $u=wx$ for some $w \in V$ with $w \not=x$. Set $t=\deg(u)$ and $N_G(u)=\{z_1,\dots,z_t\}$. Then $z_iu=z_iwx \in X \cap D$ for all $i$. For all $i$, the edges $\{z_ix,w\}$ and $\{x,z_iw\}$ are distinct since $x \notin \{z_ix,w\}$ but have the same weight $z_iwx$. Since $z_iwx \not= z_jwx$ for $i \not=j$, it follows from Proposition~\ref{prop fibers} that $|X \cap D| \le |E(G)|-\deg(u)$ as desired.
\end{proof}

\begin{proposition}\label{leaf} If $|X \cap D|=|E(G)|$, then any edge $\{u,v\}$ not contained in $V \cap P$ is of the form $\{x,x^2\}$ with $x \in V \cap P$ and  $x^2$ a leaf with unique neighbor $x$. 
\end{proposition}
\begin{proof} By Proposition~\ref{prop fibers}, the hypothesis $|X \cap D|=|E(G)|$ implies that \emph{distinct edges have distinct weights}. Let $\{u,v_1v_2\}$ be an edge with $v_1,v_2 \in V$. Thus $uv_1v_2 \in X$, so that $\{u,v_1v_2\}$, $\{v_1,uv_2\}$ and $\{v_2,uv_1\}$ are all edges in $G$ with same weight $uv_1v_2$. Hence these edges coincide, so that $u=v_1=v_2$ and the edge is $\{u,u^2\}$. Thus $u^3 \in X$. Now if $u$ were not primitive, say if $u=u_1u_2$ with $u_1,u_2 \in V$, then $u_1^3u_2^3 \in X$, and this would yield at least two distinct edges with same weight, e.g. $\{u_1, u_1^2u_2^3\}$ and $\{u_1^2, u_1u_2^3\}$. Hence $u \in V \cap P$, as claimed. Finally, let $v \in V$ be a neighbor of $u^2$. Then $u^2v\in X$, yielding two edges with same weight, namely $\{u,uv\}$ and $\{u^2,v\}$. Hence $\{u,uv\}=\{u^2,v\}$, implying $u=v$. Thus $N_G(u^2)=\{u\}$, as claimed.
\end{proof}

\section{Proof of main theorem}\label{section main}

Let $S$ be a numerical semigroup in multiplicative notation, arising from a classical numerical semigroup $S_0 \subseteq \N$ via the isomorphism~\eqref{iso}. The following notation will be used throughout Section~\ref{section main}. 

\begin{notation}\label{nota} The symbols $m,c,q,\rho,P,D,L,X$ usually associated to $S_0$ will also denote the corresponding objects in $S$ transported from $S_0$ via~\eqref{iso}. Further, we denote $G(S)=G=(V,E)$ the graph associated to $S$, and we set
\begin{equation}\label{n,k,nu}
n=|V|,\,\, k=\vm(G),\,\, \nu=\nu(G), \,\, \lambda=\lambda(G).
\end{equation}
\end{notation}
Note that by definition, we have $\lambda \le k \le n$. This section is devoted to proving Theorem~\ref{thm main}. The proof is divided into several cases and subcases depending mainly on the values of $k$ and $\lambda$. Recall that Wilf's conjecture has been shown to hold when $|P| \le 3$ or $q \le 3$, in~\cite{FGH} and~\cite{E}, respectively. Therefore, throughout the proof, we freely assume $|P| \ge 4$ and $q \ge 4$, even though these hypotheses may be dispensed of in most subcases.

\subsection{A reduction}
We first reduce the proof of Theorem~\ref{thm main} to the case $\tau(X) \le 2q-1$.

\begin{lemma} If Wilf's conjecture holds in case $\tau(X) \le 2q-1$, then Wilf's conjecture holds in case $|P| \ge m/3$.
\end{lemma}

\begin{proof} We have $W(S) = |P||L|-c=|P||L|-qm+\rho$. Assume $|P| \ge m/3$. 

\smallskip
\noindent
\textbf{Case I.} Assume $|L| \ge 3q$. Then $|P||L| \ge (m/3)(3q)=mq = c+\rho$. Therefore $W(S) \ge \rho$ and we are done.  

\smallskip
\noindent
\textbf{Case II.} Assume $|L| \le 3q-1$. Since $|L|=q+\tau(X)$, it follows that $\tau(X) \le 2q-1$. Since Wilf's conjecture is assumed to hold in this case, the proof is complete.
\end{proof}

\begin{proposition}\label{k le 4} If $\tau(X) \le 2q-1$ and $q \ge 4$, then $k \le 4$.
\end{proposition}
\begin{proof}
We have $2q-1 \ge \tau(X) \ge k(q-1)/2$. If $k \ge 5$, then $2q-1 \ge 5(q-1)/2$, implying $3 \ge q$, contrary to our assumption $q \ge 4$.
\end{proof}

Thus, we need only examine the cases $k=0,1,2,3,4$ to complete the proof of Theorem~\ref{thm main}, i.e. that $W(S) \ge 0$ in all cases under consideration. We start with $0 \le k \le 2$.

\subsection{Proof in cases $k=0,1,2$}

\vspace{0.25cm}
\noindent
\fbox{\textbf{Case $k=0$}.} Then $E=\emptyset$ and so $|X \cap D|=0$. Hence $W(S) \ge |P|\tau(X)+\rho \ge 0$.

\vspace{0.25cm}
\noindent
\fbox{\textbf{Case $k=1$}.} Then $G$ consists of exactly one loopy vertex, so $n=k=|X \cap D|=1$. Hence $\tau(X) \ge (q-1+\nu)/2$, yielding
\begin{eqnarray*}
W(S) & = & |P|\tau(X)-|X \cap D|q+\rho \\
& \ge & 4(q-1+\nu)/2-q+\rho \\
& \ge & q-2+2\nu+\rho,
\end{eqnarray*}
and so $W(S) \ge 2$ since $q \ge 4$ by assumption.

\vspace{0.25cm}
\noindent
\fbox{\textbf{Case $k=2$}.} Then $n \ge 2$ and $G$ has at most two loops, i.e. $0 \le \lambda \le 2$. By Theorem~\ref{thm formula}, we have $\tau(X) \ge q-1+\nu/2+(n-2)$, whence
\begin{equation}\label{eq k=2}
W(S) \ge |P|(q-1+\nu/2+(n-2))-|X \cap D|q+\rho.
\end{equation}

\smallskip
Assume first $|X \cap D| \le 3$. Then using $|P| \ge 4$, we have
\begin{eqnarray*}
W(S) & \ge & 4(q-1+\nu/2+(n-2))-3q+\rho \\
& = & q +4(n-3) + 2\nu + \rho.
\end{eqnarray*}
Since $n \ge 2$ and $q \ge 4$, this yields $W(S) \ge 0$ and we are done.

\smallskip
Assume now $|X \cap D| \ge 4$. Then $n \ge 3$. 

$\bullet$ The case $\lambda = 2$ cannot occur here since it would imply $n = 2$.

$\bullet$ If $\lambda = 1$, let $x \in V$ be the sole loopy vertex. Since $k < 3$, all true edges are incident to $x$. Thus all edges of $G$ are of the form $\{x,u\}$ with $u \in V$, and $|E|=|V|=n$. Since $x$ is of largest degree, namely $n$, it follows that $x \in V \cap P$ by Corollary~\ref{V subset P}. Since all edges are pairwise adjacent, all edge weights are distinct, whence $|X \cap D|=|E|=n$ by Proposition~\ref{prop fibers}. Hence $V \cap D \subseteq \{x^2\}$ by Proposition~\ref{leaf}. It follows that $|V \cap P| \ge n-1$, whence $|P| \ge n$ by Lemma~\ref{P and V}. Plugging the above information on $|X \cap D|$ and $|P|$ into \eqref{eq k=2}, we get
\begin{eqnarray*}
W(S) & \ge & n(q-1+\nu/2+(n-2))-nq+\rho \\
& = & n(n-3) + n\nu/2 + \rho,
\end{eqnarray*}
and we are done since $n \ge 3$.

$\bullet$ Finally, if $\lambda = 0$, then since $|E| \ge 4$, $G$ must be a star at a vertex $x$ with at least $3$ legs. Hence $x \in V \cap P$. Since $x$ is nonloopy, we have $x^2 \notin X$. The same argument as above, using that all edges of $G$ are of the form $\{x,u\}$ with $u \in V \setminus \{x\}$, yields $|X \cap D|=|E|=n$ and $V \cap D = \emptyset$ here. Hence $|P| \ge n+1$, yielding
\begin{eqnarray*}
W(S) & \ge & (n+1)(q-1+\nu/2+(n-2))-nq+\rho \\
& = & q+ n(n-3) + n\nu/2 + \rho.
\end{eqnarray*}
This concludes the proof in case $k=2$.

\subsection{Proof in case $k=3$}

We start with a general remark on loopy graphs $H$ with $\vm(H)=3$. 

\begin{lemma} Let $H$ be a loopy graph such that $\vm(H)=3$. Then $\lambda(H) \ge 1$, and either $K_3 \subset H \subseteq LK_3$, or else all true edges of $H$ share a common vertex.
\end{lemma}
\begin{proof} Since $\vm(H)$ is odd, it follows that $H$ has at least one loop. Since $\vm(H) < 4$, any two true edges are adjacent. Therefore, either $H$ contains a triangle, in which case $|V(H)|=3$ and $1 \le \lambda(H) \le 3$, or else all true edges of $H$ share a common vertex and $1 \le \lambda(H) \le 2$.
\end{proof}

Let us go back to our graph $G=G(S)$. We have $1 \le \lambda \le k= 3 \le n$. In the present case, it follows from Theorem~\ref{thm formula} that
\begin{equation}\label{eq k=3}
\tau(X) \ge (3(q-1)+\nu)/2+(n-3).
\end{equation}
We start with an easy particular case.

\begin{proposition} If $k=3$ and $|X \cap D| \le 4$, then $W(S) \ge 0$.
\end{proposition}
\begin{proof} As usual, we assume $|P|, q \ge 4$. By \eqref{eq k=3} we have $\tau(X) \ge 3(q-1)/2$. Hence
\begin{eqnarray*}
W(S) & \ge & |P|3(q-1)/2-4q+\rho \\
& \ge & 6(q-1)-4q+\rho \\
& = & 2(q-3)+\rho \\
& \ge & 2 + \rho. \qedhere
\end{eqnarray*}
\end{proof}

Thus, from now on in this section, we assume $|X \cap D| \ge 5$, whence in particular $|E(G)| \ge 5$.

\vspace{0.2cm}
$\bullet$ \textbf{Case $\lambda=3$.} Then $n=3$ and hence $5 \le |X \cap D| \le |E| \le 6$. By \eqref{eq k=3} we have
$\tau(X) \ge (3(q-1)+\nu)/2$, and so
\begin{eqnarray*}
W(S) & = & |P|\tau(X)-|X \cap D|q+\rho \\
& \ge & |P|(3(q-1)+\nu)/2-|X \cap D|q+\rho \\
& \ge & 6(q-1)+2\nu-|X \cap D|q+\rho.
\end{eqnarray*}

\noindent
$\circ$ Assume first $|X \cap D|=6$. Then $|E|=6$, so that $G$ is isomorphic to $LK_3$ and so all six edges are active. (See Definition~\ref{active}.) Moreover, all edge weights are distinct since $|X \cap D|=|E|$ here. The above inequalities imply
$$
W(S) \ge -6+2\nu+\rho.
$$

-- If $\nu=0$, then all six edges of $G$ are weak, whence $\rho \ge 6$ by Corollary~\ref{rho}. It follows that $W(S) \ge 0$ and we are done. 

-- If $\nu=1$ then all edges of $G$, except exactly one loop, are weak. Therefore $\rho \ge 5$, whence $W(S) \ge 1$. 

-- If $\nu=2$, then since $\nu < 3=k$, all three matchings of $G=LK_3$ have a weak edge. Hence $\rho \ge |E_0(G)| \ge 3$. It follows that $W(S) \ge -6+4+3=1$. 

-- Finally, if $\nu=3$ then $W(S) \ge \rho$ and we are done.

\medskip
\noindent
$\circ$ Assume now $|X \cap D|=5$. Then $|E|=5$ or $6$. We now have
\begin{equation}\label{X cap D = 5}
W(S) \ge q-6+2\nu+\rho.
\end{equation}
Moreover, since $G$ coincides here with either $LK_3$ or $LK_3$ minus a true edge, all edges of $G$ are active as easily seen. 

\smallskip
-- If $\nu=0$, all active edges are weak, whence $\rho \ge 4$. Hence \eqref{X cap D = 5} implies $W(S) \ge 2$ and we are done.

-- If $\nu \ge 1$ then \eqref{X cap D = 5} implies $W(S) \ge \rho$ and we are done.

\vspace{0.2cm}
$\bullet$ \textbf{Case $\lambda=2$.} Let $x_1,x_2$ denote the two loopy vertices. At the very least, besides its two loops, $G$ has one true edge adjacent to exactly one of the loopy vertices, say $x_1$. Now, either $G$ is contained in the graph with the edge $\{x_1,x_2\}$ plus pendant edges incident to $x_1$, or else $G$ is contained in $LK_3$ minus one loop, in which case $n=3$ and $|E(G)| \le 5$.

\smallskip
$\circ$ Assume first that $G$ is contained in the graph with the edge $\{x_1,x_2\}$ plus $n-2$ pendant edges incident to $x_1$. Among the $n$ vertices, at most two belong to $V \cap D$. Hence $|V \cap P| \ge n-2$, so that $|P| \ge n-1$ by Lemma~\ref{P and V}, and more precisely $|P| \ge \max(n-1,4)$. We have $|E(G)| \le 3+(n-2)=n+1$, so that $|X \cap D| \le n+1$. By \eqref{eq k=3}, it follows that
$$
W(S)  \ge  \max(n-1,4)((3(q-1)+\nu)/2+(n-3))-(n+1)q+\rho.
$$

\smallskip
$-$ If $n \ge 5$, we get
\begin{eqnarray*}
W(S) & \ge & (n-1)((3(q-1)+\nu)/2+(n-3))-(n+1)q+\rho \\
& \ge & (n-5)(q-1)/2+(n-1)(\nu/2+n-4)-2+\rho \\
& \ge & 4(\nu/2+1)-2+\rho \\
& = & 2\nu+2+\rho.
\end{eqnarray*}

\smallskip
$-$  If $3 \le n \le 4$, and using $|P| \ge 4$, we get
\begin{eqnarray*}
W(S) & \ge & 4((3(q-1)+\nu)/2+(n-3))-(n+1)q+\rho \\
& = & 6(q-1)+2\nu+4(n-3)-(n+1)(q-1)-(n+1)+\rho \\
& = & (5-n)(q-4)+2\nu+6-4+\rho \\
& \ge & 2\nu+2+\rho.
\end{eqnarray*}

\smallskip
$\circ$ Assume now that $G$ is contained in $LK_3$ minus one loop. Then $n=3$ and $|X \cap D| \le |E(G)| \le 5$. Moreover, as easily seen by inspection, at least $4$ edges of $G$ are active. 

-- If $\nu=0$, all active edges are weak, whence $\rho \ge 4$. Hence, with $|X \cap D| \le 5$, it follows from the above that $W(S) \ge 2+\rho$ and we are done. 

-- Assume now $\nu \ge 1$. By \eqref{eq k=3}, we have
$$
\tau(X) \ge (3(q-1)+\nu)/2.
$$
It follows that
\begin{eqnarray*}
W(S) & \ge & 4((3(q-1)+\nu)/2-5q+\rho \\
& = & 6(q-1) +2\nu -5q +\rho \\
& = & q-6 +2\nu +\rho \\
& \ge & \rho
\end{eqnarray*}
since $q \ge 4$ and $\nu \ge 1$.

\vspace{0.2cm}
$\bullet$ \textbf{Case $\lambda=1$.} Then $G$ contains one loopy vertex $x$ and one nonincident true edge. If $G$ contains a triangle, then $|E| \le 4$ since $k=3$, as easily seen. This is incompatible with our current assumption $|X \cap D| \ge 5$. 

Therefore $G$ is triangle-free. Hence $G$ consists of the loopy vertex $x$ and a star $T$ centered at a distinct vertex $y$. Since $|E| \ge 5$ by our current assumption, $T$ has at least $3$ pendant edges. And if $T$ is connected to $x$, then the connecting edge is between $y$ and $x$, for otherwise we would have $k \ge 4$. In any case, we have $|E| \le n+1$.

We claim that $V \subset P$. First $y \in V \cap P$ since it has maximal degree. We also have $x \in V \cap P$. For otherwise, since $x$ is loopy, we have $x^2 \in X \cap D$, whence any proper factor of $x$ would also be a loopy vertex in $G$ by Lemma~\ref{factor of loopy}, contradicting $\lambda=1$. The remaining vertices are all of degree 1 and connected to $y$, thus they form an antichain for divisibility. Hence, if any such vertex $z$ pertained to $V \cap D$, it would be a monomial in $x,y$ of length at least $2$. Now by Lemma~\ref{nonloopy}, $z$ cannot be divisible by $y$. Hence $z$ is equal to or divisible by $x^2$. Thus $yx^2 \in X$, implying $xy \in V$ and connected to $x$. But this is impossible since $N_G(x) \subseteq \{x,y\}$.

By the above and Lemma~\ref{P and V}, it follows that $|P| \ge n+1$. Using $|X \cap D| \le |E| \le n+1$ as shown earlier, we have
\begin{eqnarray*}
W(S) & \ge & |P|\tau(X)-|X \cap D|q + \rho \\
& \ge & (n+1)\tau(X)-(n+1)q + \rho \\
& \ge & (n+1)(\tau(X)-q) + \rho.
\end{eqnarray*}
But $\tau(X) > q$, since $\tau(X) \ge 3(q-1)/2$ and $q \ge 4$. Hence $W(S) \ge \rho \ge 0$.

\bigskip
The proof of the main theorem in the particular case $k=3$ is now complete.

\subsection{Proof in case $k = 4$}

By Proposition~\ref{k le 4}, the value $k=4$ is the largest admissible one for $k=\vm(G)$ under the assumption $\tau(X) \le 2q-1$. 

\smallskip
Then $n \ge 4$, and the general bound $\tau(X) \ge (k(q-1)+\nu)/2+(n-k)$ yields
\begin{eqnarray*}
\tau(X) & \ge & 2(q-1)+\nu/2+(n-4) \\
& = & 2(q-3)+\nu/2+n.
\end{eqnarray*}
This puts strong restrictions on $n$ and $\nu$.
\begin{lemma} Assume $\tau(X) \le 2q-1$ and $k=4$. Then $n \in \{4,5\}$ and $\nu \le 2$. If $n=5$, then $\nu = 0$ and $\tau(X)=2q-1$.
\end{lemma}
\begin{proof} We have $2(q-3)+\nu/2+n \le \tau(X) \le 2q-1$. Hence $\nu/2+n \le 5$. It follows that $n \le 5$ and that $\nu \le 2$ since $n \ge 4$. If $n=5$, then $\nu=0$ and the above bounds on $\tau(X)$ yield $2(q-3)+5 \le \tau(X) \le 2q-1$, whence $\tau(X)=2q-1$.
\end{proof}

\subsubsection{The subcase $k=4, n=5$}

Throughout this section, we fix the following values of the various parameters and refer to these hypotheses as the \emph{current case}:
\begin{equation}\label{n=5 k=4}
n=|V(G)|=5, \,\, k=\vm(G)=4, \,\, \tau(X) \le 2q-1.
\end{equation}

Then $\nu=0$ and $\tau(X)=2q-1$ as seen above. In particular, the former implies that \emph{all active edges are weak}. This will imply useful lower bounds on $\rho=qm-c$ and hence on $W(S)$.

We shall need an upper bound on the number of edges of $G$, actually valid in a general graph-theoretic setting.

\begin{proposition}\label{bound 10} Let $H=(V,E)$ be a loopy graph. If $|V|=5$ and $\vm(H)=4$, then $|E| \le 10$. 
\end{proposition}

\begin{proof} Set $V=V_1 \sqcup V_2$, where $V_1$ is the set of loopy vertices and $V_2 = V \setminus V_1$. Let $E=E_1 \sqcup E_2 \sqcup E_{12}$, where $E_1$ is the set of edges of the induced subgraph $H[V_1]$, $E_2$ is the edge set of $H[V_2]$ and $E_{1,2}=[V_1,V_2]$, the set of edges from $V_1$ to $V_2$. We further denote $H_1=H[V_1]$, $H_2=H[V_2]$ and $H_{1,2}$ the bipartite graph with edge set $E_{1,2}$. 

\smallskip
The proof proceeds by fixing the loop number $\lambda=\lambda(H)=|V_1|$ and letting it assume all possible values from $\vm(H)=4$ to $0$. 

\smallskip
The case $\lambda(H) = 4$ is impossible. For otherwise, since $V_2$ would consist of a single nonisolated nonloopy vertex $y_1$, there would be a true edge incident with $y_1$ and a loopy vertex $x_1 \in V_1$. But then, that edge and the three loops at the other three vertices in $V_1$ would constitute a matching touching $5$ vertices, contrary to the hypothesis $k=4$.

\smallskip
\underline{Assume $\lambda(H) = 3$}. We claim $|E| \le 8$. Set $V_1=\{x_1,x_2,x_3\}$, $V_2=\{y_1,y_2\}$. Since $\vm(H_1)=3$, we must have $\vm(H_2) \le 1$, whence $\vm(H_2) = 0$ since $H_2$ has no loops. Thus $y_1,y_2$ are not neighbours in $H$, i.e. $|E_2|=0$. Up to renumbering of $V_1$, we may assume $x_1 \in N_H(y_1)$. We claim then that $N_H(y_1)=N_H(y_2)=\{x_1\}$. Indeed, since $y_2$ is not isolated, it must have a neighbour in $V_1$. But if $y_2$ had a neighbor other than $x_1$, say $x_2$, then the edges $\{x_1,y_1\}$, $\{x_2,y_2\}$ and the loop at $x_3$ would yield $\vm(H)=5$, contrary to the hypothesis. Therefore $N_H(y_2)=\{x_1\}$. By symmetry, we get $N_H(y_1)=\{x_1\}$ as well.
Thus $|E_{1,2}|=2$. Since $|E_1| \le 6$, we conclude $|E| \le 8$ in the present case. The case $|E| = 8$ is uniquely realized, up to isomorphism, by the following loopy graph:

\bigskip
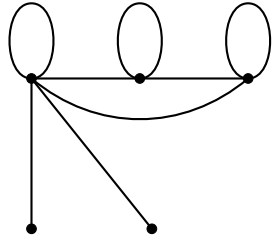
\begin{center}
\begin{tikzpicture}[scale=0.8]

\coordinate (A1) at (4,2.5); 
\coordinate (A2) at (4+1.8,2.5); 
\coordinate (A3) at (4+2*1.8,2.5); 

\coordinate (B2) at (4,0);
\coordinate (B3) at (6,0);

\draw [fill] (A1) circle (0.08);
\draw [fill] (A2) circle (0.08);
\draw [fill] (A3) circle (0.08);

\draw [fill] (B2) circle (0.08);
\draw [fill] (B3) circle (0.08);

\draw[thick ] (A1) -- (A2);
\draw[thick ] (A2) -- (A3);

\draw[thick ] (A1) -- (B2);
\draw[thick ] (A1) -- (B3);

\draw[thick ] (A1) to [out=180,in=180] ++(0,1.25) to [out=0,in=0] ++(0,-1.25);
\draw[thick ] (A2) to [out=180,in=180] ++(0,1.25) to [out=0,in=0] ++(0,-1.25);
\draw[thick ] (A3) to [out=180,in=180] ++(0,1.25) to [out=0,in=0] ++(0,-1.25);

\draw[thick ] (A1) to [out=270+50,in=270-50] (A3);

\end{tikzpicture}
\end{center}
\bigskip

\underline{Assume $\lambda(H) = 2$}. We claim $|E| \le 9$. Indeed, as easily seen, there are exactly three isomorphism classes of edge-maximal loopy graphs $H$ with the given parameters. These classes have $6$, $7$ and $9$ edges, respectively:

\bigskip
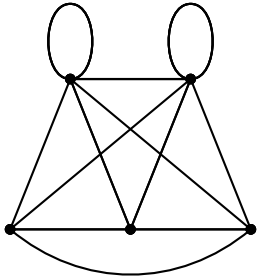
\begin{tikzpicture}[scale=0.8]

\hskip-0.3cm;

\coordinate (A1) at (4,2.5); 
\coordinate (A2) at (6,2.5); 

\coordinate (B1) at (3,0);
\coordinate (B2) at (5,0);
\coordinate (B3) at (7,0);

\draw [fill] (A1) circle (0.08);
\draw [fill] (A2) circle (0.08);
\draw [fill] (B1) circle (0.08);
\draw [fill] (B2) circle (0.08);
\draw [fill] (B3) circle (0.08);

\draw[thick ] (A1) -- (A2);

\draw[thick ] (B1) -- (B2);
\draw[thick ] (B2) -- (B3);
\draw[thick ] (B1) to [out=270+50,in=270-50] (B3);

\draw[thick ] (A1) to [out=180,in=180] ++(0,1.25) to [out=0,in=0] ++(0,-1.25);
\draw[thick ] (A2) to [out=180,in=180] ++(0,1.25) to [out=0,in=0] ++(0,-1.25);


\hskip5cm;

\coordinate (A1) at (4,2.5); 
\coordinate (A2) at (6,2.5); 

\coordinate (B1) at (3,0);
\coordinate (B2) at (5,0);
\coordinate (B3) at (7,0);

\draw [fill] (A1) circle (0.08);
\draw [fill] (A2) circle (0.08);
\draw [fill] (B1) circle (0.08);
\draw [fill] (B2) circle (0.08);
\draw [fill] (B3) circle (0.08);

\draw[thick ] (A1) -- (A2);
\draw[thick ] (A1) -- (B2);
\draw[thick ] (A2) -- (B2);

\draw[thick ] (B1) -- (B2);
\draw[thick ] (B2) -- (B3);

\draw[thick ] (A1) to [out=180,in=180] ++(0,1.25) to [out=0,in=0] ++(0,-1.25);
\draw[thick ] (A2) to [out=180,in=180] ++(0,1.25) to [out=0,in=0] ++(0,-1.25);


\hskip5cm;

\coordinate (A1) at (4,2.5); 
\coordinate (A2) at (6,2.5); 

\coordinate (B1) at (3,0);
\coordinate (B2) at (5,0);
\coordinate (B3) at (7,0);

\draw [fill] (A1) circle (0.08);
\draw [fill] (A2) circle (0.08);
\draw [fill] (B1) circle (0.08);
\draw [fill] (B2) circle (0.08);
\draw [fill] (B3) circle (0.08);

\draw[thick ] (A1) -- (A2);

\draw[thick ] (A1) -- (B1);
\draw[thick ] (A1) -- (B2);
\draw[thick ] (A1) -- (B3);

\draw[thick ] (A2) -- (B1);
\draw[thick ] (A2) -- (B2);
\draw[thick ] (A2) -- (B3); 

\draw[thick ] (A1) to [out=180,in=180] ++(0,1.25) to [out=0,in=0] ++(0,-1.25);
\draw[thick ] (A2) to [out=180,in=180] ++(0,1.25) to [out=0,in=0] ++(0,-1.25);

\end{tikzpicture}

\medskip
\underline{Assume $\lambda(H) = 1$}. We claim $|E| \le 8$. Indeed, the unique isomorphism class of edge-maximal loopy graphs $H$ with the given parameters is the following one, with $8$ edges:

\bigskip
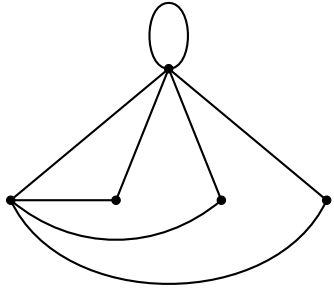
\begin{center}
\begin{tikzpicture}[scale=0.7]

\coordinate (A1) at (5,2.5); 

\coordinate (B1) at (2,0);
\coordinate (B2) at (4,0);
\coordinate (B3) at (6,0);
\coordinate (B4) at (8,0);

\draw [fill] (A1) circle (0.08);

\draw [fill] (B1) circle (0.08);
\draw [fill] (B2) circle (0.08);
\draw [fill] (B3) circle (0.08);
\draw [fill] (B4) circle (0.08);

\draw[thick ] (A1) -- (B1);
\draw[thick ] (A1) -- (B2);
\draw[thick ] (A1) -- (B3);
\draw[thick ] (A1) -- (B4);

\draw[thick ] (B1) -- (B2);

\draw[thick ] (B1) to [out=270+50,in=270-50] (B3);
\draw[thick ] (B1) to [out=270+25,in=270-25] (B4);

\draw[thick ] (A1) to [out=180,in=180] ++(0,1.25) to [out=0,in=0] ++(0,-1.25);

\end{tikzpicture}
\end{center}

\medskip
\underline{Assume $\lambda(H) = 0$}. Then $|E| \le 10$. Indeed, the complete graph $K_5$ is the unique edge-maximal simple graph with the given parameters.
\end{proof}

\smallskip
Let us go back to our graph $G=G(S)=(V,E)$. Since $|X \cap D| \le |E|$, the above result implies $|X \cap D| \le 10$. We start with a reduction to the case $|X \cap D| \in \{8,9\}$.

\begin{proposition}\label{X cap D le 7} In the current case~\eqref{n=5 k=4}, if either $|X \cap D| \le 7$, or $V \subset P$, or $|X \cap D| \ge 10$, then $S$ satisfies Wilf's conjecture. 
\end{proposition}

\begin{proof} \hspace{0mm}

\noindent
$\bullet$ Assume $|X \cap D| \le 7$. We have $W(S) \ge |P|(2q-1)-7q+\rho$. Our assumptions $|P|,q \ge 4$ further yield $W(S) \ge 4(2q-1)-7q+\rho=q-4+\rho \ge \rho$ and we are done.

\smallskip
\noindent
$\bullet$ Assume $V \subset P$. Then $|P| \ge |V|+1=6$. Hence, using $|X \cap D| \le 10$, we have
\begin{eqnarray*}
W(S) & \ge & |P|\tau(X)-|X \cap D|q + \rho \\
& \ge & 6(2q-1)-10q + \rho \\
& = & 2q-6+\rho.
\end{eqnarray*}
Since $q \ge 4$ in the current case, we get $W(S) \ge 2 +\rho$ and we are done.

\smallskip
\noindent
$\bullet$ Assume $|X \cap D| \ge 10$. By Proposition~\ref{bound 10}, we have $|E| \le 10$. Whence $|E|=10$ since $|E| \ge |X \cap D| \ge 10$. Moreover, it follows from the proof of that Proposition that the only case where $|E|=10$ is $G=LK_5$. Since $G$ is regular, it follows from Corollary~\ref{V subset P} that $V \subset P$. Thus $S$ satisfies Wilf's conjecture by the previous case.
\end{proof}

\smallskip
We next assume $|V \cap D| = 1$.  

\begin{proposition} In the current case~\eqref{n=5 k=4}, if $|V \cap D|=1$ then $S$ satisfies Wilf's conjecture.
\end{proposition}
\begin{proof} The hypotheses imply $|V \cap P|=4$, whence $|P| \ge 5$. Moreover, by the previous result, we may assume $|X \cap D| \le 9$. Then
\begin{eqnarray*}
W(S) & \ge & |P|\tau(X)-|X \cap D|q + \rho \\
& \ge & 5(2q-1)-9q + \rho \\
& = & q-5+\rho.
\end{eqnarray*}
Since $\nu=0$, and since there is a vertex-maximal matching touching $4$ vertices, it follows that there at least two active weak edges. Corollary~\ref{rho} then implies $\rho \ge 1$, and we conclude $W(S) \ge 0$ as desired.
\end{proof}

\medskip
It remains to treat the case $|V \cap D| \ge 2$ and $|X \cap D| \in \{8,9\}$. From here, we again proceed by descending values of $\lambda(G)$ from $4$ to $0$.
The case $\lambda=4$ is impossible in the present context.

\smallskip
\underline{Assume $\lambda = 3$}. Let $x_1,x_2,x_3$ be the loopy vertices and $y_1,y_2$ the nonloopy ones. We have seen that $|E| \le 8$ in this case. But since $|X \cap D| \ge 8$, it follows that $|X \cap D|=|E|=8$. This only way to achieve this, up to isomorphism, is that $G$ contains $LK_3$ on the vertices $x_1,x_2,x_3$ with $y_1,y_2$ linked to $x_1$. (See corresponding picture in the proof of Proposition~\ref{bound 10}.) We have $x_1 \in P$ since it is of highest degree. Since $|V \cap D| \ge 2$ by assumption, it follows from Proposition~\ref{leaf} that $V \cap D$ consists of leaves, each of the form $x^2$ with $x \in V \cap P$ as unique neighbor. Therefore $V \cap D = \{y_1,y_2\}$, and since both have $x_1$ as unique neighbor, this implies $y_1=y_2=x_1^2$, an absurdity since $y_1,y_2$ are distinct. Hence the present case, namely $n=5$, $k=4$, $|X \cap D| \ge 8$, $|V \cap D| \ge 2$ and $\lambda = 3$, cannot occur.

\smallskip
\underline{Assume $\lambda = 2$}. Let $x_1,x_2$ be the loopy vertices and $y_1,y_2,y_3$ the nonloopy ones. We have seen that $|E| \le 9$ in this case. If $|E|=9$, then $G$ is the join between $LK_2$ and $\overline{K_3}$, i.e.
$$
G = LK_2 \vee \overline{K_3}
$$
as pictured here:
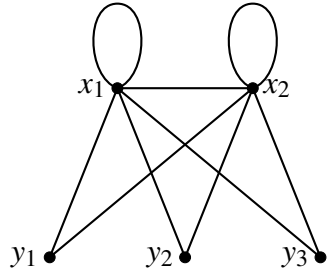
\begin{center}
\begin{tikzpicture}[scale=0.9]
\coordinate (A1) at (4,2.5); 
\coordinate (A2) at (6,2.5); 

\coordinate (B1) at (3,0);
\coordinate (B2) at (5,0);
\coordinate (B3) at (7,0);

\draw [fill] (A1) circle (0.08) node[left] {$x_1$};
\draw [fill] (A2) circle (0.08) node[right] {$x_2$};
\draw [fill] (B1) circle (0.08) node[left] {$y_1$};
\draw [fill] (B2) circle (0.08) node[left] {$y_2$};
\draw [fill] (B3) circle (0.08) node[left] {$y_3$};

\draw[thick ] (A1) -- (A2);

\draw[thick ] (A1) -- (B1);
\draw[thick ] (A1) -- (B2);
\draw[thick ] (A1) -- (B3);

\draw[thick ] (A2) -- (B1);
\draw[thick ] (A2) -- (B2);
\draw[thick ] (A2) -- (B3); 

\draw[thick ] (A1) to [out=180-20,in=180] ++(0,1.25) to [out=0,in=20] ++(0,-1.25);
\draw[thick ] (A2) to [out=180-20,in=180] ++(0,1.25) to [out=0,in=20] ++(0,-1.25);
\end{tikzpicture}
\end{center}
Incidentally, note that this graph realizes the first occurrences of $W_0(S) < 0$. (See \cite{EF} for more details.) We further assume $|V \cap D| \ge 2$. We claim that 
\begin{equation}\label{claim}
V \cap D =\{y_1,y_2,y_3\}=\{x_1^2,x_1x_2,x_2^2\}.
\end{equation}
Indeed, by Corollary~\ref{V subset P}, the $x_i$ belong to  $V \cap P$ since they have maximal degree $5$, and the $y_i$ constitute an antichain for divisibility since they all have degree $2$. Hence the vertices in $V \cap D$ are monomials in $x_1,x_2$. By symmetry, we may assume $y_1 \in V \cap D$ and $y_1=x_1u$ for some $u \in V$. Since $\{x_1,y_1\} \in E$,  it follows that $x_1^2u \in X$. Hence $x_1^2 \in V \cap D$. Up to symmetry again, we may assume $y_1=x_1^2$. Since $\{x_2,y_1\} \in E$, we have $x_1^2x_2 \in X$, whence $x_1x_2 \in V \cap D$. Say $y_2=x_1x_2$. Since $\{x_2,y_2\} \in E$, it follows that $x_1x_2^2 \in X$. Hence $x_2^2 \in V \cap D$, implying $y_3=x_2^2$. This proves~\eqref{claim}, as claimed. Now, even though $|E|=9$ here, Proposition~\ref{prop fibers} implies $|X \cap D| \le 7$ since two pairs of edges have the same weight, namely 
\begin{eqnarray*}
\wt(\{x_1,x_1x_2\}) & = & \wt(\{x_1^2,x_2\}), \\
\wt(\{x_2,x_1x_2\}) & = & \wt(\{x_2^2,x_1\}).
 \end{eqnarray*}
Therefore this case is settled by Proposition~\ref{X cap D le 7}.

Assume now $|E|=|X \cap D|=8$, and still $|V \cap D| \ge 2$ of course. Then $G$ is obtained by suppressing an edge from the graph $LK_2 \vee \overline{K_3}$ above. By Proposition~\ref{leaf}, the vertices in $V \cap D$ must all be of degree one. However, in $G$, at most one vertex has  degree one as easily seen. Therefore this case is impossible.

\smallskip
\underline{Assume $\lambda = 1$}. Then $|E|=|X \cap D|=8$ again. As seen above, $G$ is the join $LK_1 \vee T$ of a loop $LK_1$ with a claw $T$. However, this case is again made impossible by Proposition~\ref{leaf} since there are no vertices of degree $1$.

\smallskip
\underline{Assume $\lambda = 0$}. Again, we may assume $|X \cap D| \in \{8,9\}$ and $|V \cap D| \ge 2$. We have $G \subseteq K_5$ since it has $5$ vertices and no loops. 

The case $G=K_5$ is impossible, for it would imply $V \subset P$, contrary to our hypotheses. Hence $|E| \in \{8,9\}$ and $G$ is obtained by removing $1$ or $2$ edges from $K_5$. 

If $|X \cap D|=|E|$, then Proposition~\ref{leaf} implies that the vertices in $V \cap D$ have degree $1$. But $G$ has no vertices of degree less than $2$, so this case is impossible. 

It remains to consider the case $|X \cap D|=8$, $|E|=9$. Thus $G$ is $K_5$ minus one edge, i.e. $G = K_3 \vee \overline{K_2}$. Its degree distribution is $(3,3,4,4,4)$. Hence $|V \cap P|=3$, $|V \cap D|=2$. Set $V \cap P=\{x_1,x_2,x_3\}$, $V \cap D=\{y_1,y_2\}$. Then $y_1,y_2$ are monomials in $x_1,x_2,x_3$. Assume $y_1$ is divisible by $x_j$ for some $j$, so $y_1=x_jv$ for some $v \in V$. Since $y_1,x_j$ are neighbors, it follows that $x_jy_1 \in X$, whence $x_j^2v \in X$, whence $x_j^2 \in X$. Therefore $x_j$ is a loopy vertex, in contradiction with the hypothesis $\lambda = 0$. Hence this case is impossible as well.

This completes the verification of Wilf's conjecture in case $k=4$, $n=5$ and $\tau(X) \le 2q-1$. 

\subsubsection{The subcase $k=4, n=4$}

Throughout this section, the \emph{current case} is given by the following hypotheses:
\begin{equation}\label{n=4 k=4}
n=|V(G)|=4, \,\, k=\vm(G)=4, \,\, \tau(X) \le 2q-1.
\end{equation}
This implies
\begin{equation}
\tau(X) \ge 2(q-1)+\nu/2
\end{equation}
and $\nu \le 2$ in this context, as seen above. We have $|E| \le 10$, the number of edges of $LK_4$.

\begin{proposition}\label{k=4 n=4} In the current case~\eqref{n=4 k=4}, if either $|X \cap D| \le 6$ or $V \subset P$, then $S$ satisfies Wilf's conjecture. 
\end{proposition}
\begin{proof} As above, we freely assume $|P|,q \ge 4$. 

\smallskip
\noindent
$\bullet$ Assume $|X \cap D| \le 6$. Then
\begin{eqnarray*}
W(S) & \ge & |P|(2(q-1)+\nu/2)-6q+\rho \\
& \ge & 8(q-1)+2\nu-6q+\rho \\
& = & 2q-8+2\nu +\rho \\
& \ge & 2\nu +\rho
\end{eqnarray*} 
and we are done.

\smallskip
\noindent
$\bullet$ Assume $V \subset P$. Then $|P| \ge 5$ here. Thus
\begin{eqnarray*}
W(S) & \ge & 5\big(2(q-1)+\nu/2\big) - |X\cap D|q + \rho \\
& = & (10-|X \cap D|)q+5\nu/2-10+ \rho.
\end{eqnarray*}
We now examine separately the cases $|X \cap D|=10, 9, 8, 7$.

\smallskip
$\circ$ If $|X \cap D|=10$, then $|E|=10$ and $G=LK_4$. Then
$$
W(S) \ge 5\nu/2-10+\rho.
$$
Since $G=LK_4$, all $10$ edges are active.

-- If $\nu=0$, then all edges are weak, i.e. $E=E_0^+$. We have $\rho \ge \wt(E_0)$, and since $\wt$ is a bijection here, this implies $\rho \ge 10.$ Hence $W(S) \ge 0$ if $\nu=0$. 

-- If $\nu=1$, then exactly one vertex is touched by a normal edge. Hence all edges are weak except one loop. It follows that $\rho \ge 9$, whence $W(S) \ge 5/2-10+9$, implying $W(S) \ge 2$.

-- Finally, if $\nu=2$, then at most $2$ vertices are touched by a normal edge. Hence at most $3$ edges are normal, and so at least $7$ edges are weak. It follows that $\rho \ge 7$. Hence $W(S) \ge 5-10+7=2$. This completes the case $|X \cap D|=10$.

\smallskip
$\circ$ If $|X \cap D|=9$, then $W(S) \ge q-10+5\nu/2+\rho$. Then here also, each edge is active. 

-- If $\nu=0$, then all edges are weak, hence $\rho \ge 9$. Thus $W(S) \ge -6+9=3$.

-- If $\nu=1$, then exactly one loop is normal. Hence there are at least $8$ weak active edges, so that $\rho \ge 8$. Thus $W(S) \ge -6+5/2 +8$, implying $W(S) \ge 5$.

-- Finally, if $\nu=2$, then at most $2$ vertices are touched by normal edges, hence at most $3$ edges are normal. Hence there are at least $6$ active weak edges, implying $\rho \ge 6$. Hence $W(S) \ge 5$ and we are done for the case $V \subset P$, $|X \cap D| = 9$.

\smallskip
$\circ$ If $|X \cap D|=8$, then $W(S) \ge 2q-10+5\nu/2+\rho \ge -2+5\nu/2+\rho$. Then $G$ is $LK_4$ with at most $2$ missing edges. Then, as easily seen by examining the various possibilities for $G$, it is straightforward to check that $G$ contains at least $7$ active edges in each case. 

-- If $\nu = 0$, then the above implies $\rho \ge 7$, and so $W(S) \ge -2+\rho \ge 5$.
 
-- If $\nu \ge 1$, then $W(S) \ge -2+5\nu/2+\rho \ge 1+\rho$ and we are done.

\smallskip
$\circ$ If $|X \cap D|=7$, then $W(S) \ge 3q-10+5\nu/2+\rho \ge 2+5\nu/2+\rho$ and we are done. This completes the proof of the proposition.
\end{proof}

Having settled the case $|V \cap D|=0$, we now tackle the case $|V \cap D|=1$.

\begin{proposition}\label{k=4 n=4 V cap D = 1} In the current case~\eqref{n=4 k=4}, if $|V \cap D| =1$ then $S$ satisfies Wilf's conjecture. 
\end{proposition}
\begin{proof} Set $V \cap D = \{u\}$. It follows from Proposition~\ref{card V cap D = 1} that $u=x^2$ with $x \in P$ as its sole neighbor. Hence $u$ is a nonloopy vertex and $\deg(u)=1$. The latter implies $|E| \le 7$. Since $|X \cap D| \le |E|$ and the case $|X \cap D| \le 6$ has already been settled, it remains to examine the case $|X \cap D| = |E|=7$. Therefore $G$ consists of $LK_3$ with $x$ as one of the vertices, to which a pendant edge is attached with endvertex $u=x^2$:

\medskip
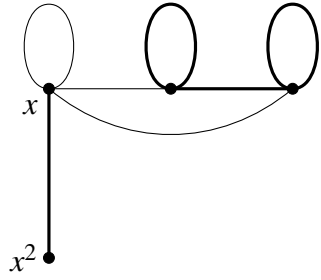
\begin{center}
\begin{tikzpicture}[scale=0.9]

\coordinate (A1) at (4,2.5); 
\coordinate (A2) at (4+1.8,2.5); 
\coordinate (A3) at (4+2*1.8,2.5); 

\coordinate (B2) at (4,0);

\draw [fill] (A1) circle (0.08) node [below left] {$x$};
\draw [fill] (A2) circle (0.08);
\draw [fill] (A3) circle (0.08);

\draw [fill] (B2) circle (0.08) node [left] {$x^2$};

\draw[ ] (A1) -- (A2);
\draw[very thick ] (A2) -- (A3);

\draw[very thick ] (A1) -- (B2);

\draw[ ] (A1) to [out=180,in=180] ++(0,1.25) to [out=0,in=0] ++(0,-1.25);
\draw[very thick ] (A2) to [out=180,in=180] ++(0,1.25) to [out=0,in=0] ++(0,-1.25);
\draw[very thick ] (A3) to [out=180,in=180] ++(0,1.25) to [out=0,in=0] ++(0,-1.25);

\draw[ ] (A1) to [out=270+50,in=270-50] (A3);

\end{tikzpicture}
\end{center}
\medskip
Note that $G$ has exactly $4$ active edges, the thicker ones in the picture. We have 
$$W(S) \ge 8(q-1)+2\nu-7q+\rho=q-8+2\nu+\rho.$$

-- If $\nu=0$ then all active edges of $G$ are weak. Since $\wt$ is a bijection here, it follows that $\rho \ge 4$. Hence $W(S) \ge 0$, as desired.

-- If $\nu=1$ then all active edges are weak, except for one normal loop. It follows that $\rho \ge 3$ and that $W(S) \ge 1$.

-- If $\nu \ge 2$ then $W(S) \ge \rho$ since $q \ge 4$.
\end{proof}

\medskip
It remains to consider the cases $|V\cap D| = 2,3$.

\begin{proposition}\label{k=4 n=4 V cap D = 1} In the current case~\eqref{n=4 k=4}, if $|V \cap D| \ge 2$ then $S$ satisfies Wilf's conjecture. 
\end{proposition}

\begin{proof} Assume first $|V \cap D| = 2$. Set $V \cap P = \{x_1,x_2\}$ and $V \cap D = \{u_1,u_2\}$. Thus $u_1,u_2$ are monomials in $x_1,x_2$. We claim that $|X \cap D| \le 6$. Indeed, as $V$ is a downset, the only possibilities up to symmetry are 
$$
\{u_1,u_2\} \,=\, \{x_1^2,x_1x_2\}, \, \{x_1^2,x_2^2\}, \, \{x_1^3,x_1^2\}.$$ 
Now, since all proper factors of the elements of $X \cap D$ are vertices by Lemma~\ref{V and X cap D}, the corresponding only possibilities for $X \cap D$ are 
$$
\{x_1^3, x_1^2x_2, x_1^2, x_1x_2, x_2^2\}, \, \{x_1^3, x_2^3, x_1^2, x_1x_2, x_2^2\}, \, \{x_1^4, x_1^3, x_1^2, x_1^2, x_1x_2, x_2^2\},
$$
respectively, as is straightforward to check. For instance, if $\{u_1,u_2\}=\{x_1^2,x_1x_2\}$, then $x_1x_2^2$ cannot belong to $X \cap D$ since its proper factor $x_2^2$ is not in $V$. This concludes the proof of the claim, and hence of the case $|V \cap D|=2$ by Proposition~\ref{k=4 n=4}.

\smallskip
Assume finally $|V \cap D|=3$. Set $V \cap P=\{x\}$. Then again, since $V$ is a downset and made of monomials in $x$, it follows that $V \cap D = \{x^2, x^3, x^4\}$. Therefore $X \cap D=\{x^2, x^3, x^4, x^5\}$ and we are done again.
\end{proof}

\medskip

This concludes our proof of Theorem~\ref{thm main}. We close this section with a straightforward consequence.

\begin{corollary}\label{cor 12} Wilf's conjecture holds for all numerical semigroups of multiplicity $m \le 12$.
\end{corollary}
\begin{proof} Let $S$ be a numerical semigroup of multiplicity $m \le 12$. If $|P| \le 3$ then $S$ satisfies Wilf's conjecture by~\cite{FGH}. If $|P| \ge 4$ then $|P| \ge m/3$ since $m \le 12$, and we conclude with Theorem~\ref{thm main}.
\end{proof}

\begin{remark}
Corollary~\ref{cor 12} has just been improved with a \emph{verification of Wilf's conjecture up to multiplicity $m \le 18$}, by computer calculations with a specially developed algorithm based on the Kunz polytope and polyhedral geometry~\cite{BGOW}.
\end{remark}

\section{Equivalence of numerical semigroups}\label{section equivalence}

In this section, we investigate the range of the map $S \mapsto G(S)$ and we briefly consider its fibers.

\subsection{Realizability}

Given any loopy graph $G$, is there a numerical semigroup $S$ such that $G(S)$ is isomorphic to $G$? The answer is given below.

We first recall a notation from~\cite{E}. If $x_1,\dots,x_n,t$ are positive integers, we denote by $\vs{x_1,\dots,x_n}_t$ the numerical semigroup defined as follows:
$$
\vs{x_1,\dots,x_n}_t = \vs{x_1,\dots,x_n} \cup [t, \infty[.
$$
This construction makes sense even if the $x_i$ are not globally coprime. Note that the conductor $c$ of $\vs{x_1,\dots,x_n}_t$ satisfies $c \le t$.

\begin{theorem} Let $G=(V,E)$ be a loopy graph. Then there exist infinitely many numerical semigroups $S$ such that $G(S)$ is isomorphic to $G$. 
\end{theorem}
\begin{proof} Set $n=|V|$. Take $m$ sufficiently large, and choose any integer sequence $x_1,\dots,x_n$ satisfying the following two conditions:
\begin{itemize}
\item $m/3 \le x_1 < \dots < x_n < (m-1)/2$, \vspace{-0.2cm}
\item the $x_i+x_j$ are pairwise distinct. 
\end{itemize}
Then the $n+\binom{n+1}{2}$ elements of the set
$$\{x_1,\dots,x_n\} \cup \{x_i+x_j\ \mid 1 \le i \le j \le n\}$$
are \emph{pairwise distinct mod $m$}. This is because $x_i \in [m/3,(m-1)/2[$ and $x_i+x_j \in [2m/3,m-1[$ for all $i,j$. Let
$$
S_0 = \vs{m,m+x_1,\dots,m+x_n}_{2m}.
$$

The above directly implies $G(S_0)=LK_n$. To obtain $G$ itself, we need only  erase in $LK_n$ those edges not belonging to $G$. For each edge $\{m+x_i, m+x_j\}$ to be erased, it suffices to add to $S_0$ the new generator $m+x_i+x_j$. This will yield $S$ such that $G(S)=G$. Details are left as an exercise to the reader.
\end{proof}

For instance, here are realizations of the complete loopy graph $LK_n$ as $G(S)$ for infinitely many numerical semigroups $S$. For a subset $A$ in $\Z$ or $\Z/m\Z$, we denote $2A=A+A=\{a+b \mid a,b \in A\}$.
\begin{example} The graph $LK_3$ is realized by the numerical semigroup $$S=\vs{m,m+1,m+3,m+7}_{2m}$$ with the condition $2(m+7) \le 2m+(m-1)$, i.e. with $m \ge 15$. Setting $A=\{m+1,m+3,m+7\}$, and computing $A \cup 2A$ in $\Z/m\Z$, we have
$$
A \cup 2A \equiv \{1,3,7\} \sqcup \{2,4,8,6,10,14\} \bmod m.
$$
Since $m \ge 15$, these $9$ elements are nonzero and pairwise distinct mod $m$. Moreover, $2A \subseteq [c,c+m[=[2m,3m-1[$. Hence $G(S)$ is the loopy-complete triangle.
\end{example}

\begin{example} More generally, the graph $LK_n$ is realized by the numerical semigroup $$S=\vs{m,m+1,m+3, \dots, m+2^n-1}_{2m}$$ with the condition $2(m+2^n-1) \le 2m+(m-1)$, i.e. with $m \ge 2^{n+1}-1$. Setting $A=\{m+1,m+3,\dots m+2^n-1\}$, and computing $A \cup 2A$ in $\Z/m\Z$, we have
$$
A \cup 2A \equiv \{1,3,\dots,2^n-1\} \sqcup \{2,4,8, \dots, 2^{n+1}-2\} \bmod m.
$$
\end{example}

\subsection{Graph-equivalence}
We now briefly consider the fibers of the map $S \mapsto G(S)$.

\begin{definition} Let $S,S'$ be two numerical semigroups. We say that $S,S'$ are \emph{graph-equivalent} if their associated graphs $G(S), G(S')$ are isomorphic.
\end{definition}

For instance, the class of numerical semigroups $S$ such that $G(S)=\emptyset$ is well known. It coincides with the set of so-called \emph{maximal embedding dimension numerical semigroups}, i.e. those for which $e=m$, where $e=|P|$ is the embedding dimension and $m$ is the multiplicity. Indeed, we have 
$$
|P|=m \iff P = X \sqcup \{m\} \iff X \cap D=\emptyset,
$$
where $P,X$ are the sets of primitive and nonzero Apéry elements of $S$, respectively.

\bigskip
The following tables give, for all $1 \le g \le 20$, 
\begin{itemize}
\item the number $n_g$ of numerical semigroups of genus $g$, \vspace{-0.2cm}
\item the number $\gamma_g$ of \emph{equivalence classes} of numerical semigroups of genus $g$.
\end{itemize}

\medskip
\begin{tabular}{||c||c|c|c|c|c|c|c|c|c|c|c|c|c|c|}
\hline
$g$ & 1 & 2 & 3 & 4 & 5 & 6 & 7 & 8 & 9 & 10 & 11 & 12 \\
\hline \hline
$n_g$ & 1 & 2 & 4 & 7 & 12 & 23 & 39 & 67 & 118 & 204 & 343 & 592 \\
\hline
$\gamma_g$ & 1 & 1 & 2 & 3 & 4 & 6 & 11 & 15 & 27 & 41 & 66 & 115 \\
\hline
\end{tabular}
\bigskip

\begin{tabular}{||c||c|c|c|c|c|c|c|c|}
\hline
$g$ & 13 & 14 & 15 & 16 & 17 & 18 & 19 & 20 \\
\hline \hline
$n_g$ & 1001 & 1693 & 2857 & 4806 & 8045 & 13467 & 22464 & 37396  \\
\hline
$\gamma_g$ & 190 & 322 & 569 & 1014 & 1761 & 3107 & 5475 & 9621 \\
\hline
\end{tabular}

\bigskip

Those values of $\gamma_g$ were obtained using the function \texttt{IsomorphicGraphQ} in \emph{Mathematica 10}. Needless to say, it would be very interesting to determine the long-term behavior of the sequence $\gamma_g$. 

\smallskip
For instance, for $g=7$, the $39$ numerical semigroups of genus $7$ regroup into $\gamma_7=11$ equivalence classes. The eleven nonisomorphic loopy graphs arising this way are the following ones: the empty graph, the two loopy graphs with $1$ edge, the five loopy graphs with $2$ edges, and three more loopy graphs with $3$ edges, namely

\bigskip
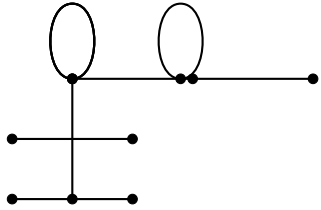
\begin{tikzpicture}[scale=0.8]

\hskip0.1cm

\coordinate (A1) at (0,2); 
\coordinate (A2) at (0+1.8,2); 

\coordinate (B2) at (0,0);

\draw [fill] (A1) circle (0.08);
\draw [fill] (A2) circle (0.08);

\draw [fill] (B2) circle (0.08);

\draw[thick ] (A1) -- (B2);

\draw[thick ] (A1) to [out=180,in=180] ++(0,1.25) to [out=0,in=0] ++(0,-1.25);
\draw[thick ] (A2) to [out=180,in=180] ++(0,1.25) to [out=0,in=0] ++(0,-1.25);

\hskip4cm

\coordinate (A1) at (0,2); 

\coordinate (B2) at (2,2);
\coordinate (B3) at (4,2);

\draw [fill] (A1) circle (0.08);

\draw [fill] (B2) circle (0.08);
\draw [fill] (B3) circle (0.08);

\draw[thick ] (A1) -- (B2);
\draw[thick ] (B2) -- (B3);

\draw[thick ] (A1) to [out=180,in=180] ++(0,1.25) to [out=0,in=0] ++(0,-1.25);

\hskip6cm

\coordinate (A1) at (0,2); 

\coordinate (B1) at (-1,1);
\coordinate (B2) at (1,1);

\coordinate (C1) at (-1,0);
\coordinate (C2) at (1,0);

\draw [fill] (A1) circle (0.08);
\draw [fill] (B1) circle (0.08);
\draw [fill] (B2) circle (0.08);
\draw [fill] (C1) circle (0.08);
\draw [fill] (C2) circle (0.08);

\draw[thick ] (B1) -- (B2);
\draw[thick ] (C1) -- (C2);

\draw[thick ] (A1) to [out=180,in=180] ++(0,1.25) to [out=0,in=0] ++(0,-1.25);

\end{tikzpicture}

\bigskip

We conclude this paper with a question. Can one show \emph{a priori} that if a numerical semigroup $S$ satisfies Wilf's conjecture, then so do all equivalent numerical semigroups $S' \sim S$? For instance, the less dense $G(S)$ is, the easier one may expect checking Wilf's conjecture on $S$ will be. At any rate, the proofs in this paper show that the properties of the graphs $G(S)$ for the numerical semigroups $S$ under consideration play a central role towards this  endeavor.

\bigskip 
\textbf{Acknowledgments.} The author wishes to thank Manuel Delgado and Jean Fromentin for many useful discussions related to this work.


\medskip

\noindent
{\small
\textbf{Author's address:}

\noindent
Shalom Eliahou\textsuperscript{a,b}

\noindent
\textsuperscript{a}Univ. Littoral C\^ote d'Opale, EA 2597 - LMPA - Laboratoire de Math\'ematiques Pures et Appliqu\'ees Joseph Liouville, F-62228 Calais, France\\
\textsuperscript{b}CNRS, FR 2956, France\\
\textbf{e-mail:} eliahou@lmpa.univ-littoral.fr
}

\end{document}